\documentclass{amsart}

\newtheorem{theorem}{Theorem}[section]

\newtheorem{corollary}[theorem]{Corollary}

\theoremstyle{definition}
\newtheorem{definition}[theorem]{Definition}
\newtheorem{example}[theorem]{Example}

\theoremstyle{remark}
\newtheorem{remark}[theorem]{Remark}

\usepackage{braket,amsfonts}
\usepackage{tikz}
\usepackage{pst-node}
\usepackage{tikz-cd} 
\usepackage{array}

\usepackage{pgfplots}

\usepackage{graphicx,epstopdf}
\usepackage{subcaption}   

\usepackage[utf8]{inputenc} 
\usepackage[T1]{fontenc}    
\usepackage{hyperref}       
\usepackage{url}            
\usepackage{booktabs}       
\usepackage{amsfonts}       
\usepackage{nicefrac}       
\usepackage{microtype}      
\usepackage{amsmath, mathtools, amssymb}
\usepackage{bbm}            
\usepackage{algorithm}
\usepackage[noend]{algpseudocode}
\usepackage{todonotes}
\usepackage{multirow}
\usepackage{siunitx}
\usepackage{enumitem}
\usepackage{xr}

\usepackage{amsopn}

\usepackage{xspace}
\usepackage{bold-extra}
\usepackage[most]{tcolorbox}

\colorlet{texcscolor}{blue!50!black}
\colorlet{texemcolor}{red!70!black}
\colorlet{texpreamble}{red!70!black}
\colorlet{codebackground}{black!25!white!25}

\newcommand\Algphase[1]{%
\vspace*{-.3\baselineskip}\Statex\hspace*{\dimexpr-\algorithmicindent-2pt\relax}
\Statex\hspace*{-\algorithmicindent}\textbf{#1}%
\vspace*{-.7\baselineskip}\Statex\hspace*{\dimexpr-\algorithmicindent-2pt\relax}
}

\newtheorem{prop}{Proposition}

\def\SOP{\mathcal{S}^{\dagger}}
\def\SOPA{\mathcal{S}_\theta}
\def\sopa{s_{\theta}}
\def\EQOP{\mathcal{G}}
\def\SU{\mathcal{U}}

\def\SA{\mathcal{A}}

\def\LOSS{\mathcal{L}}

\def\R{\mathbb{R}}
\def\IND#1{\mathbbm{1}_{{#1}}}

\def\Id{\mathbf{I}}
\numberwithin{equation}{section}

\def\PSU{U}
\def\PSV{V}
\def\BU{\mathbf{u}}
\def\BV{\mathbf{v}}
\def\BA{\mathbf{a}}
\def\SC{\Psi}

\graphicspath{{./Figures/}{./}}

\begin{document}

\title[Generative learning for parametric PDEs]{Generative diffusion learning for parametric partial differential equations}


\author{Ting Wang}
\address{Booz Allen Hamilton Inc, McLean, VA, and\newline
Physical Modeling and Simulation Branch, DEVCOM Army Research Laboratory, Aberdeen Proving Ground, MD}
\email{wang\_ting@bah.com}

\author{Petr Plech\'a\v{c}}
\address{Department of Mathematical Sciences, University of Delaware, Newark, DE}
\email{plechac@udel.edu}
\thanks{Research of P.P. was supported by the U.S. Army Research Office Grant W911NF2220234.
Computations in this research were supported in part through the use of DARWIN computing system: A Resource for Computational and Data-intensive Research at the University of Delaware and in the Delaware Region, which is supported by NSF under Grant Number: 1919839.}

\author{Jaroslaw Knap}
\address{Physical Modeling and Simulation Branch, DEVCOM Army Research Laboratory, Aberdeen Proving Ground, MD}
\email{jaroslaw.knap.civ@army.mil}


\date{\today}


\begin{abstract}
We develop a class of data-driven generative models that approximate the solution operator for parameter-dependent partial differential equations (PDEs). We introduce a probabilistic formulation of the operator learning problem based on denoising diffusion probabilistic models (DDPM), which enables learning the input-output mapping between problem parameters and PDE solutions. We generalize DDPM to a supervised setting, where the solution operator is represented by a family of conditional distributions. This probabilistic formulation, combined with DDPM, naturally provides uncertainty quantification through confidence intervals for the learned solutions. Moreover, the framework is directly applicable to learning solution operators from noisy data sets. We evaluate the computational performance of our method against Fourier Neural Operators (FNO) and show that it achieves comparable accuracy while additionally recovering the noise magnitude in data sets corrupted by additive noise.
\end{abstract}

\maketitle

%
%

\section{Introduction}
In many scientific and engineering applications, it is essential to handle computational models involving random input parameters. Often, this task is computationally challenging as it relies on repeated solves of the underlying mathematical model for distinct values of input parameters. 
It is therefore crucial to develop fast and reliable methods that solve the problems for a range of input parameters. 

\subsection{Parametric partial differential equations}
In the context of parameter-dependent partial differential equations (PDEs) the reliance of a solution on input parameters can be
conveniently captured by the following abstract PDE
\begin{equation}\label{eqn:random-PDE}
\begin{split}
\EQOP(u(x); a(x)) &= 0,\qquad
\text{in~} D \subset \mathbb{R}^d, \\
\end{split}
\end{equation}
defined by the mapping $\EQOP:\SU\times\SA \to \mathbb{R}$.
Here $\SA$ and $\SU$ are suitable spaces of functions over the domain $D$ for which the problem is well-posed.
The input $a = a(\cdot) \in \mathcal{A}$ takes values in the input function space $\SA$
and $u = u(\cdot) \in \mathcal{U}$ is the corresponding solution taking values in the output function space $\SU$. 
A typical example where the proposed approach will be applied is the case when $a$ is generated by a random field distributed according to a probability measure $\nu_a$.
We give specific examples of
PDEs later when discussing various benchmarks in Section~\ref{sec:experiments}. For now, we tacitly assume that the choice of $\EQOP$, $\SA$, and $\SU$ defines a well-posed problem.
In other words, we assume that there exists a unique input-to-output solution operator
\[
\SOP: \SA \to \SU, \quad u = \SOP(a)\,,
\]
such that $\EQOP(u(x); a(x)) = 0$ in $D$.

\subsection{An overview of operator learning}
The overarching goal of the operator learning problem is to find an operator 
$\SOPA : \mathcal{A} \to \mathcal{U}$, parameterized by the parameters $\theta$, and approximating $\SOP$ in a certain sense.
Once $\SOPA$ is learned, solutions to~\eqref{eqn:random-PDE} can be readily evaluated for a range of inputs, thus avoiding repeated solves of the equation for different $a$.

We emphasize that the function spaces $\SA$ and $\SU$ are infinite-dimensional by definition. However, in the numerical context we only have access to the pair of input function $a \in \mathcal{A}$ and output function $u \in \mathcal{U}$ in some finite-dimensional approximation spaces $\mathbb{R}^{d_{\mathcal{A}}}$ and $\mathbb{R}^{d_{\mathcal{U}}}$, respectively. 
To this end, we write $\BA \in \mathbb{R}^{d_{\mathcal{A}}}$ and $\BU \in \mathbb{R}^{d_{\mathcal{U}}}$
for the finite-dimensional vectors obtained by restricting $a$ and $u$ to 
$\mathbb{R}^{d_{\mathcal{A}}}$ and $\mathbb{R}^{d_{\mathcal{U}}}$, respectively.
We assume that there is a unique finite-dimensional mapping such that 
\[
s^{\dagger}: \mathbb{R}^{d_{\mathcal{A}}} \to \mathbb{R}^{d_{\mathcal{U}}}, \quad \BU = s^{\dagger}(\BA)\,.
\]
We denote the interpolation maps 
\[
\mathcal{I}_{\mathcal{A}}: \mathbb{R}^{d_{\mathcal{A}}} \to \mathcal{A}, \quad 
\mathcal{I}_{\mathcal{U}}: \mathbb{R}^{d_{\mathcal{U}}} \to \mathcal{U}
\]
that interpolate from finite-dimensional latent spaces to infinite-dimensional spaces of functions, 
and the restriction maps
\[
\mathcal{R}_{\mathcal{A}}: \mathcal{A} \to \mathbb{R}^{d_{\mathcal{A}}}, \quad 
\mathcal{R}_{\mathcal{U}}: \mathcal{U} \to \mathbb{R}^{d_{\mathcal{U}}}
\]
that project from infinite-dimensional spaces of functions to finite-dimensional latent spaces.
Hence, we have
\begin{eqnarray*}
a = \mathcal{I}_{\mathcal{A}}(\BA)\,, \quad  
u = \mathcal{I}_{\mathcal{U}}(\BU)\,,\quad\quad
\BA = \mathcal{R}_{\mathcal{A}}(a)\,, \quad 
\BU = \mathcal{R}_{\mathcal{U}}(u)\,.
\end{eqnarray*}
The input function $a$ is a realization of a random field assumed to be distributed according to a
known measure $\nu_a$ on the infinite-dimensional space 
$\mathcal{A}$ (see Section~\ref{sec:experiments} for examples of $\nu_a$). 
We denote 
$\nu_{\BA}$ the
pushforward measure $\mathcal{R}^\#_\mathcal{A}\nu_a$ 
of the finite dimensional approximation $\BA = \mathcal{R}_\mathcal{A}(a)$. 
The schematic diagram for operator learning is presented in Figure~\ref{fig:op-diagram}, which suggests two different routes to learn the infinite-dimensional operator: 1) infinite-dimensional approaches that learn directly on functional spaces
\begin{equation}\label{eqn:direct}
\mathcal{S}_{\theta} \approx \mathcal{S}^{\dagger} 
: \mathcal{A} \to \mathcal{U}\,,
\end{equation}
and 2) finite-dimensional approaches that learn the composition 
\begin{equation}\label{eqn:indirect}
\mathcal{I}_{\mathcal{U}} \circ s_{\theta} \circ \mathcal{R}_{\mathcal{A}}
\approx 
\mathcal{S}^{\dagger}: \mathcal{A} \to \mathcal{U}
\,,
\end{equation}
where $s_\theta:\mathbb{R}^{d_{\mathcal{A}}}\to\mathbb{R}^{d_{\mathcal{U}}}$ is the finite-dimensional operator that maps $\BA$ and $\BU$.

\begin{figure}
    \centering
\begin{tikzcd}
a \in \mathcal{A} \arrow[r, "\mathcal{S}_{\theta} \approx \mathcal{S}^{\dagger}"] \arrow[d,  "\mathcal{R}_{\mathcal{A}}"] 
& u \in \mathcal{U} \arrow[d, shift left=2, "\mathcal{R}_{\mathcal{U}}"] \\
\BA \in \mathbb{R}^{d_{\mathcal{A}}} \arrow[r, "s_{\theta} \approx s^{\dagger}"] \arrow[u, shift left=2, "\mathcal{I}_{\mathcal{A}}"]
& \BU \in \mathbb{R}^{d_{\mathcal{U}}} \arrow[u, "\mathcal{I}_{\mathcal{U}}"]
\end{tikzcd}
    \caption{The commutative diagram for operator learning}
    \label{fig:op-diagram}
\end{figure}

The two approaches~\eqref{eqn:direct} and~\eqref{eqn:indirect} provide a natural dichotomy for operator learning methods.
The finite-dimensional approach projects functions into a finite-dimensional representation via discretization or basis expansion and approximates the operator with a mapping between these finite-dimensional spaces. 
Most conventional methods for operator learning, e.g., polynomial chaos~\cite{xiu2002wiener}, stochastic collocation~\cite{babuvska2007stochastic} and reduced basis method~\cite{devore2017theoretical}, all belong to this category. 
More recently, neural network-based approximation~\cite{zhu2018bayesian, khoo2021solving}  and kernel-based regression~\cite{batlle2023kernel} have also been extensively explored for finite-dimensional operator learning.

In contrast, the infinite-dimensional (functional) approach employs specialized neural architectures to directly approximate operators as mappings between function spaces, thereby enabling mesh- and resolution-independent representations. These methods can be broadly classified into two categories: data-driven and physics-informed.
The data-driven approach assumes access to a data set of input-output pairs $\mathcal{D} = \{(\BA^{(n)}, \BU^{(n)})\}_{n=1}^{M}$,
and aims to learn the parametric approximation 
$\SOPA$, such that the empirical data loss
\begin{equation}\label{eqn:C-loss}
\LOSS_{\text{DD}}(\theta) = \frac{1}{{M}}\sum_{n=1}^{M} \|u^{(n)} - \SOPA(a^{(n)})\|^2
\end{equation}
is minimized with respect to parameters $\theta$.
Recently, various deep learning architectures have been designed in order to solve~\eqref{eqn:C-loss}~efficiently. 
State-of-the-art examples include fully convolution networks~\cite{FCN}, PCA networks~\cite{PCANN, hesthaven2018non}, multi-wavelet methods~\cite{gupta2021multiwavelet}, 
turbulent flow nets~\cite{wang2020towards}, 
deep operator nets (DeepONets)~\cite{ONET, POD-ONET}, Fourier neural operators (FNO)~\cite{kovachki2021neural, FNO, GNO}. Among these, DeepONets leverage the universal approximation theorem of operators~\cite{chen1995universal}, while FNOs directly parameterize the integral kernel in the Fourier space~\cite{kovachki2021universal}. Both have received significant attention due to their mesh-independent nature. Despite their success across challenging parametric PDE problems, data-driven methods remain highly data-intensive.
The physics-informed approach relies on the direct knowledge of the underlying PDE to learn $\SOPA$ and utilizes the physics-informed empirical loss in the minimal residual form
\begin{equation}\label{physics-informed-loss}
    \LOSS_{\mathrm{PI}}(\theta)=\frac{1}{{M}}\sum_{n=1}^{M} \|\mathcal{G}(\SOPA(a^{(n)}); a^{(n)})\|^2\,.
\end{equation}
Since the loss only involves the input $a$ and not the solution (output) $u$, 
the physics-informed approach is, in principle, ``data-free'' by design 
~\cite{PINO, patel2021physics, wang2021learning, wang2022stochastic, zhu2019physics}.
Nevertheless, optimization of the physics-informed loss remains challenging and consequently limits its applicability to relatively simple problems~\cite{sun2020surrogate, wang2021understanding,  wang2022}. 

Finally, we emphasize that all of the aforementioned methods learn an operator $\SOPA$ that is inherently deterministic, in the sense that a given input $a$ is always mapped to a unique output $u = \SOPA(a)$. 

\subsection{Denoising diffusion probabilistic model} 
Recently, DDPM has outperformed traditional deep generative models in various machine learning tasks thanks to its ability to generate high quality samples from complex, high dimensional distributions~\cite{dhariwal2021diffusion, ho2020denoising, sohl2015deep, song2019generative}.
Inspired by sampling techniques from non-equilibrium thermodynamics, DDPM utilizes the time reversal property for a class of parameterized Markov chains to learn a sequence of latent distributions that eventually converge to the data distribution~\cite{song2020score, ho2020denoising}. 
Once properly trained, the model sequentially transforms a given Gaussian noise into a sample from the data distribution.
Despite its remarkable performance in a wide range of applications, the standard DDPM trained on a given resolution fails to generalize to different resolutions~\cite{teng2023relay, rombach2022high} and therefore it is not resolution invariant by design.
It is desirable to directly model the distribution of interest in functional spaces and avoid discretization until the last possible moment~\cite{stuart2010inverse}. This viewpoint has motivated numerous recent works on extending DDPM (and more general diffusion-based models) to functional spaces, allowing natural modeling of distribution over functions~\cite{baker2024conditioning, lim2023score, lim2023score_hilbert, pidstrigach2023infinite, kerrigan2022diffusion}.
The functional space DDPM has recently been shown to be effective in various scientific applications, such as Bayesian inverse problem~\cite{baldassari2023conditional, baldassari2024taming}
and stochastic optimal control~\cite{park2024stochastic}.
See~\cite{franzese2025generative} for an excellent overview of infinite-dimensional diffusion models. 

We emphasize that our probabilistic formulation of operator learning can be easily combined with other generative models such as the normalizing flow model~\cite{papamakarios2021normalizing} and the optimal transport generative model.~\cite{albergo2022building}

\subsection{Contributions of our work.}
Building on DDPM, we introduce a probabilistic method, termed the probabilistic diffusion neural operator (PDNO), for solving the operator learning problem in a generative framework. Specifically, we reformulate operator learning as a conditional distribution learning task: given an input $a$, the goal is to learn a generative model for the conditional distribution $u|a$. This formulation enables the natural use of DDPM as the underlying sampler for $u|a$. Since PDNO maximizes the evidence likelihood while learning a family of conditional distributions, it can be directly applied to noisy data sets, where $u|a$ is no longer concentrated on a Dirac measure (i.e., the noise-free case). Furthermore, the learned standard deviation of $u|a$ provides a principled quantification of predictive uncertainty and model confidence.
Finally, based on recent works on the functional space DDPM, we discuss the extension of PDNO to the infinite-dimensional setting.
Compared with FNO and ONet, our method requires fewer regularity assumptions while leveraging the flexibility of DDPM in modeling complex distributions. As a result, PDNO is particularly well-suited for learning highly non-smooth solutions, as demonstrated in the advection PDE example in Section~\ref{sec:experiments}.

In Section~\ref{sec:framework}, we reformulate the operator learning problem as a probabilistic learning task from noisy data. Section~\ref{sec:ddpm} introduces the PDNO solver, where DDPM is generalized to the supervised setting for operator learning. In Section~\ref{sec:ddpm-functional}, we present possible extensions of DDPM to functional spaces. Section~\ref{sec:experiments} presents numerical benchmarks that demonstrate the accuracy of PDNO. For convenience, the notation used throughout the manuscript is summarized in Table~\ref{tab:notations}.

%
%

\section{A probabilistic formulation of operator learning with noisy data}\label{sec:framework}
In this section, we introduce a probabilistic framework for operator learning through maximizing the conditional data likelihood, enabling operator learning with highly noisy dataset. 
We restrict our framework to learning the finite-dimensional mapping $s_{\theta}$, although it naturally extends to the functional setting for the operator $\mathcal{S}_{\theta}$. This is because the generative model employed in the next section corresponds to an indirect approach to operator learning.

\subsection{Operator learning with noisy data}
In practice, the training data for operator learning are often obtained from solving~\eqref{eqn:random-PDE} when $\mathcal{G}$ is known, or from experiments when $\mathcal{G}$ is unknown. 
In either case, the output data are corrupted by a numerical or observational noise $\eta$ and hence it is reasonable to assume that the output is of the form
\begin{equation}\label{eqn:noisy-model}
\BU^{\eta} = s_{\theta}(\BA) + \eta,
\end{equation}
where $\eta$ is a $\mathbb{R}^{d_{\mathcal{U}}}$-valued random noise.
Note that we do not assume any particular distribution on $\eta$ but only require that 
\[
\mathbb{E}[\eta] = 0,  \quad \text{Cov}[\eta] = \lambda \Sigma
\] 
for some unknown constant $\lambda \geq 0$ and positive definite matrix $\Sigma$. 
For simplicity, we assume that the noise is independent of the input $\BA$.
Given a noisy input-output dataset $\mathcal{D}_{\eta} = \{(\BA^{(n)}, \BU^{\eta,(n)})\}_{n=1}^{M}$,
we aim to learn the finite dimensional operator $s_{\theta}: \mathbb{R}^{d_{\mathcal{A}}} \to \mathbb{R}^{d_{\mathcal{U}}}$ and additionally the distribution of the  noise $\eta$. The noisy operator learning setting
leads to a natural probabilistic formulation.
We denote $p(\BA,\BU^{\eta})$ the joint density on the product space $\mathbb{R}^{d_{\mathcal{A}}} \times \mathbb{R}^{d_{\mathcal{U}}}$ that generates $(\BA, \BU^{\eta})$, $p(\BU^{\eta}|\BA)$ is the conditional density, and $p(\BA)$ denotes the marginal.
Similarly, we write $q_{\theta}(\BA,\BU^{\eta})$ for a parametric approximation to $p(\BA, \BU^{\eta})$ and the goal is to learn a probabilistic conditional model 
$q_{\theta}(\BU | \BA)$.
Maximizing the conditional data likelihood is equivalent to maximizing the joint data likelihood
\begin{equation}\label{eqn:likelihood}
    \prod_{n=1}^{M} q_{\theta}(\BU^{\eta,(n)} | \BA^{(n)}) \propto
    q_{\theta}(\mathcal{D}_{\eta}) = \prod_{n=1}^{M} q_{\theta}(\BU^{\eta,(n)} | \BA^{(n)}) p(\BA^{(n)}).
\end{equation}
This induces the following negative log-likelihood (NLL) loss function,
\begin{equation}\label{eqn:NLL-loss}
\mathcal{L}_{\text{NLL}}(\theta) = -\sum_{n=1}^{M}\log q_{\theta}(\BU^{\eta,(n)} | \BA^{(n)})
\end{equation}
The mean of the conditional model $q_{\theta}$ provides an approximation to the finite-dimensional operator $s_{\theta}$, i.e., 
\[
s_{\theta}(\BA)=\int_{\R^{d_\mathcal{U}}} \BU^{\eta} q_{\theta}(\BU^{\eta}|\BA)\, d\BU^\eta \,.
\]
Additionally, the uncertainty of the approximation can be quantified by estimating the variance of $q_{\theta}$, i.e.,
\[
\int_{\R^{d_\mathcal{U}}} \|\BU^{\eta}-s_\theta(\BA)\|^2 q_{\theta}(\BU^{\eta}|\BA)\, d\BU^\eta 
\]
Two special cases of our probabilistic formulation are of particular interest.
\begin{itemize}
\item {\it Gaussian noise setting.} When $\lambda > 0$, $\eta$ follows a Gaussian
distribution $\mathcal{N}(0, \lambda\Sigma)$, the data set may contain repeated realizations of $\BA$ whose corresponding outputs $\BU^{\eta}$ are different due to noise. 
In this case, our probabilistic model $q_{\theta}(\BU^{\eta} | \BA)$ 
learns a class of Gaussian distributions $\mathcal{N}(s_{\theta}(\BA), \lambda\Sigma)$. 
It is important to note that the NLL loss~\eqref{eqn:NLL-loss} reduces to the data-driven loss~\eqref{eqn:C-loss} when the covariance matrix $\lambda\Sigma = \lambda \Id$ with $\lambda$ being a known constant. 
Therefore, the NLL loss function~\eqref{eqn:NLL-loss} generalizes the $L_2$ loss functions commonly used in the operator learning literature. 
From a Gaussian regression perspective, if we place a Gaussian random field prior on $s_{\theta}$, i.e., $s_{\theta} \sim \mathcal{GP}(0, K)$ with some matrix-valued kernel $K$, our probabilistic formulation reduces to the vector-valued Gaussian processes regression of~\cite{batlle2023kernel}.

\item \indent {\it Noise-free setting.} When $\lambda = 0$, for each input $\BA$, there is a unique output $\BU$ in the noise-free data set $\mathcal{D}$ that solves~\eqref{eqn:random-PDE}.
In this case, our probabilistic model $q_{\theta}(\cdot | \BA)$ aims at learning 
a class of Dirac measures $\delta_{\BU}(\cdot|\BA)$. 
However, from a theoretical perspective, the density $p(\cdot|\BA)$ of $\delta_{\BU}(\cdot|\BA)$ does not exist. Instead,
the learning task should be reformulated in terms of distributions/measures, and the goal is to approximate the Dirac measure $\delta_{\BU}(\cdot|\BA)$. 
We postpone the precise meaning of operator learning from noise-free data to Section~\ref{sec:ddpm-functional} (see Remark~\ref{remark:noise-free-OL}).
For ease of presentation, we still write symbolically $q_{\theta}(\cdot| \BA)$ for the density without reformulating our probabilistic framework in terms of measures. 
\end{itemize}

The advantage of the probabilistic approach to operator learning is two-fold:
{\it (i) Flexibility:} the probabilistic model $q_{\theta}(\cdot|\BA)$ is more flexible than deterministic models since it does not depend 
on the specific choice of the cost functional 
$\LOSS$. 
As we have discussed, the framework includes the $L_2$ loss as a special case when $\eta$ is Gaussian.
We do not assume any distribution for the conditional $q_{\theta}(\cdot|\BA)$ but rather learn it completely from the data in a generative manner. 
{\it (ii) Uncertainty evaluation:}  
In practice, the data can contain noise due to numerical/measurement errors, and similarly, the trained model can demonstrate uncertainty 
due to over-fitting or under-fitting. 
Aside from the prediction of the output, 
we are often interested in the quality of the prediction as well. 
The probabilistic model $q_{\theta}(\cdot| \BA)$ allows us to quantify the data and model uncertainty. 

%
%

\section{Conditional DDPM for operator learning}\label{sec:ddpm}
Under the proposed probabilistic framework, the specific model that we employ for learning the distribution $q_{\theta}(\cdot | \BA)$ is the DDPM, which provides a generative model for sampling from the desired distribution.
In this section, we first generalize the DDPM for operator learning in finite-dimensional space and then discuss its extensions to the functional space setting. 

\subsection{Conditional DDPM in finite dimensional spaces}
Given a set of unlabeled observations, the DDPM learns a generative model by constructing a time reversible Markov chain whose initial distribution coincides with the distribution of interest~\cite{song2020score, ho2020denoising}. 
To apply the DDPM for the conditional distribution $q_{\theta}(\cdot |\BA)$, we generalize the denoising framework to the conditional setting. 
To this end, for each input realization $\BA$, we define a complete probability space $(\Omega, \mathcal{F}, \mathbb{P}_{\BA})$\footnote{For the sake of readability, we avoid the abstract measure-theoretic treatment involving conditional expectations.}
and consider a stochastic process $\{\PSU(t; \BA)\}_{t\in [0, T]}$ satisfying the stochastic differential equation (SDE)
\begin{equation}\label{eqn:forward-sde-finite}
\begin{split}
&d\PSU(t; \BA) = -\frac{1}{2} \alpha(t) \PSU(t; \BA) \, dt + \sqrt{\alpha(t)} dW(t)\,, \\
&\PSU(0; \BA) \sim p(\cdot| \BA)\,,
\end{split}
\end{equation}
for $t \in [0, T]$, 
where $\BA \in \mathbb{R}^{d_{\mathcal{A}}}$ is an input sample in the operator learning task, $W(t)$ is a $\mathbb{R}^{d_{\mathcal{U}}}$-valued standard Brownian motion adapted to the filtration $\{\mathcal{F}_t\}$ and the  noise scheduling function $\alpha(t) \in \mathbb{R}$ is chosen such that 
$\PSU(T; \BA)$ converges to the standard normal 
$\mathcal{N}(0, \mathbf{I})$ as $T \to \infty$\footnote{Our framework allows for a generalization of $\alpha(t)$ to $\alpha(t; \BA)$, allowing the noise schedule to depend on the input parameter 
$\BA$. However, as we demonstrate in the numerical section, an input-independent noise schedule already yields satisfactory performance for the DDPM.}.
Note that we require that the distribution of the initial state $\PSU(0; \BA)$ coincides with the true conditional distribution $p(\cdot| \BA)$ from the operator learning task.
Denote $\mu_t(\cdot|\BA)$ the marginal distribution of $\PSU(t; \BA)$ and $\rho_t(\cdot | \BA)$ the marginal density of $\PSU(t; \BA)$ with respect to the Lebesgue measure $d\BU$, i.e., 
\[
\mu_t(B | \BA) = \mathbb{P}_{\BA}(\PSU(t; \BA) \in B) \,, \qquad  B \text{~is a Borel set in}~\mathbb{R}^{d_{\mathcal{U}}}\,, 
\]
and 
\[
{\mu_t(d\BU | \BA)} = \rho_t(\BU | \BA) d\BU\,.
\] 
We define the reverse process 
$\{\PSV(t;\BA)\}_{t \in [0, T]}$ by $\PSV(t;\BA) := \PSU(T-t;\BA)$. Thanks to the martingale representation theorem~\cite{ethier2009markov}, under certain conditions, the reverse process can be written as the solution to the following reverse SDE
\begin{equation}\label{eqn:reverse-SDE}
\begin{split}
&d\PSV(t; \BA) = \left(\frac{1}{2} \alpha(T-t) \PSV(t; \BA) + \alpha(T-t) \SC(T-t, \PSV(t; \BA),  \BA)\right)\, dt\\
    &\qquad \qquad\quad + \sqrt{\alpha(T-t)} d \widetilde{W}(t)\,,\\
        &\PSV(0;\BA) \sim \rho_T(\cdot | \BA)
\end{split}
\end{equation}
for $t \in [0, T]$, 
where $\Psi: [0, T] \times \mathbb{R}^{d_{\mathcal{U}}} \times \mathbb{R}^{d_{\mathcal{A}}} \to \mathbb{R}^{d_{\mathcal{U}}}$
is the score function of the marginal density, i.e.,
\begin{equation}\label{eqn:score-function}
\SC(t, \BU, \BA) := \nabla_{\BU}  \log \rho_{t} (\BU | \BA)
\end{equation}
 and $\widetilde{W}(t)$ is a $\mathbb{R}^{d_{\mathcal{U}}}$-valued standard Brownian motion adapted to a reversed filtration $\{\widetilde{\mathcal{F}}_t\}$. 
By definition, for a given input parameter $\BA$, the reverse SDE~\eqref{eqn:reverse-SDE} transports $\rho_T(\cdot | \BA) \approx \mathcal{N}(0, \mathbf{I})$ to the target distribution $\rho_0(\cdot| \BA) \approx p(\cdot | \BA)$, thereby serving as a generative model for sampling from it. 

\begin{example}
As an example, we consider the following time-homogeneous Ornstein-Uhlenbeck (OU) process
    \begin{equation}\label{eqn:OU-process}
    d\PSU(t; \BA) = -\frac{1}{2}\PSU(t; \BA) dt + \sqrt{C} dW(t)\,,
    \end{equation}
where $C \in \mathbb{R}^{d_{\mathcal{U}} \times d_{\mathcal{U}}}$ is positive definite. 
Note that, when $C$ is identity, a random-time-change $t \to \int_0^t \alpha(s) \, ds$ 
of the above SDE leads to the forward SDE~\eqref{eqn:forward-sde-finite}, \cite{oksendal2013stochastic}.
The score function associated to the OU process admits an analytical form
    \begin{equation}\label{eqn:OU-score-function}
    \begin{split}
     C\nabla_{\BU}  \log \rho_{t} (\BU | \BA) = -\frac{1}{1 - \text{e}^{-t}} \mathbb{E}_{\BA}\left[ U(t; \BA) -\text{e}^{-\frac{t}{2}} U(0; \BA)  ~|~ U(t; \BA) = \BU\right]\,, \quad t \in (0, T],
    \end{split}
    \end{equation}
where $\mathbb{E}_{\BA}$ is the expectation taken with respect to all realizations of the OU process $\{U(t; \BA)\}_{t \in [0, T]}$ conditioned on the input $\BA \in \mathbb{R}^{d_{\mathcal{A}}}$. 
The proof is presented in Appendix~\ref{app:OU-score-function} for completeness.
\end{example}

\begin{remark}
    The conditional expectation on the right-hand side of~\eqref{eqn:OU-score-function} can alternatively be used as the definition of the score function, which is more convenient to deal with the case when the density (with respect to the Lebesgue measure) does not exist. For example, when the initial distribution is Dirac, i.e., $U(0;\BA) \sim \delta_{\BU_0}(\cdot)$, the left-hand side of~\eqref{eqn:OU-score-function} is not well-defined whereas the $t \to 0^+$ limit of the right-hand side still exists. 
\end{remark}

In general, the score function is often unknown and therefore must be approximated through parameterization $\SC_{\theta}(t, \BU, \BA)$. 
After taking expectations with respect $\nu_{\BA}$ and averaging over time, a natural choice of the loss function is the score matching (SM) objective
\begin{equation}\label{eqn:SM-objective}
\mathcal{L}_{\text{SM}}(\theta) := \frac{1}{T}\int_0^T \ell_{\mathrm{SM}}(t;\theta)\, dt\,,
\end{equation}
where
\[
\ell_{\mathrm{SM}}(t;\theta):= \mathbb{E}_{\nu_{\BA}} \mathbb{E}_{\rho_t(\BU | \BA)} \left[\|\SC_{\theta}(t, \PSU(t; \BA), \BA) - \SC(t, \PSU(t; \BA), \BA)\|^2\right]\,,
\]
and $\mathbb{E}_{\rho_t(\BU | \BA)}$ is the expectation taken with respect to the density $\rho_t(\BU | \BA)$ and $\mathbb{E}_{\nu_{\BA}}$ is the expectation taken with respect to the measure $\nu_{\BA}$.
Unfortunately, the SM objective is intractable since the initial distribution is unknown. 
However, the following result provides a tractable alternative to the SM objective, 
\cite{vincent2011connection}; see Appendix~\ref{app:remainder-R} for the derivation.
\begin{theorem}\label{thm:DSM-objective-continuous}
Suppose that the true score function satisfies that
\[
\sup_{t \in (0, T]}\mathbb{E}_{\nu_{\BA}}\mathbb{E}_{\rho_{t}(\BU|\BA)} \left[\|\SC(t, \PSU(t; \BA), \BA)\|^2\right] < \infty\,.
\]
Then minimizing the time-dependent SM objective~\eqref{eqn:SM-objective} is equivalent to minimizing the following denoising score matching (DSM) objective 
\begin{equation}\label{eqn:DSM-continuous}
\mathcal{L}_{\text{DSM}}(\theta) := \frac{1}{T}\int_0^T \ell_{\text{DSM}}(t; \theta)  \, dt\,
\end{equation}
with
\[
\ell_{\text{DSM}}(t; \theta) 
:=
\mathbb{E}_{\nu_{\BA}}
\mathbb{E}_{\rho_0(\BU_0 | \BA)}
\mathbb{E}_{\rho_{t|0}(\BU | \BU_0)} \left[\|\SC_{\theta}(t, \PSU(t; \BA), \BA) - \nabla_{\BU} \log \rho_{t|0} (\PSU(t; \BA) | \PSU(0;\BA)) \|^2\right]\,,
\]
where 
$\rho_{t|0} (\BU | \BU_0)$\footnote{Note that when conditioned on $\PSU_0 = \BU_0$, $\rho_{t|0} (\BU | \BU_0) = \rho_{t|0} (\BU | \BU_0, \BA)$ and hence the omission of $\BA$ in the notation.} is the density of $\PSU(t; \BA)$ conditioned on the initial state $\PSU(0; \BA) = \BU_0$
and 
the expectation $\mathbb{E}_{\rho_0(\BU_0 | \BA)}$ is taken with respect to the initial density $\rho_0(\BU_0 | \BA)$.
\end{theorem}

For numerical implementation, an Euler-Maruyama discretization of the forward SDE with step size $\Delta t \ll 1$ leads to
a class of discrete time Markov chains $\PSU_{0:N}(\BA)$ satisfying the discrete dynamics\footnote{Here we used the fact that $1 - \frac{1}{2} \beta_n \approx \sqrt{1 - \beta_n}$ when $\Delta t \ll 1$.},
\[
\PSU_n(\BA) = \sqrt{1 - \beta_{n}}\PSU_{n-1}(\BA) 
+ \sqrt{\beta_{n}} \epsilon_n, \qquad \PSU_0(\BA) \sim p(\cdot | \BA) 
\]
for $n = 1, \ldots, N = \lfloor T/\Delta t \rfloor$, where $\epsilon_n$ are $\mathbb{R}^{d_{\mathcal{U}}}$-valued iid standard Gaussians and 
\[\beta_n := \alpha((n-1) \Delta t) \Delta t.\]
Throughout the paper we consider the path measure of $\PSU_{0:N}(\BA)$ with the
joint density $\bar{\rho}(\BU_0,\BU_1,\dots,\BU_N | \BA)$\footnote{For simplicity of notation, we do not use the notation~$\bar{\rho}_{[0:N]}(\BU_0,\BU_1,\dots,\BU_N | \BA)$ to indicate the fact that it is the path density from $0$ to $N$. The meaning of the density should be clear from the variables $\BU_0, \ldots, \BU_N$ in the density. Unless otherwise specified, the same type of notation applies to other density functions, e.g., $\bar{\rho}(\BU_n | \BU_{n-1} )$ stands for $\bar{\rho}_{n|n-1}(\BU_n | \BU_{n-1} )$}, that is,
\[
\bar{\rho}(\BU_0,\BU_1,\dots,\BU_N | \BA)\, d\BU_0 \dots d\BU_N = \mathbb{P}_{\BA}(\PSU_0(\BA) \in d\BU_0, \ldots, \PSU_N(\BA) \in d\BU_N).
\]
Consequently, for $\Delta t$ small, the above Markov chain induces the forward transition density for $\PSU_{n-1}(\BA)$ to $\PSU_n(\BA)$,
\begin{equation}\label{eqn:forward-transition}
    \bar{\rho}(\BU_n | \BU_{n-1} ) = \mathcal{N}(\sqrt{1 - \beta_{n}} \BU_{n-1}, \beta_{n} \mathbf{I}) 
\end{equation}
for $n = 1, \ldots, N$.
We emphasize that $\PSU_n(\BA)$ depends on $\BA$ only through $\PSU_{n-1}(\BA)$.
Recall that by definition, the initial density coincides with the true conditional density in the operator learning task, i.e., 
\[
\bar{\rho}(\BU_0 | \BA) = p(\BU_0| \BA)\,.
\] 
By Markov property, one can readily derive the $n$-step forward transition density (conditioned on $\BU_{0}$),
\begin{equation}\label{eqn:un_given_u0}
\bar{\rho}(\BU_n | \BU_{0}) := \mathcal{N}(\sqrt{\gamma_n}\BU_0, (1 - \gamma_n)\mathbf{I}),
\end{equation}
and hence
\[
\nabla_{\BU_n}\log \bar{\rho}(\BU_n | \BU_{0}) = -\frac{\BU_n - \sqrt{\gamma_n}\BU_0}{1 - \gamma_n}\,,
\]
where $\gamma_n = \prod_{i=1}^n (1 - \beta_n)$. 
Since $\PSU_n(\BA) - \sqrt{\gamma_n}\PSU_0(\BA) \overset{\text{d}}{=} \sqrt{1 - \gamma_n}\epsilon_n$, we have
\[
\nabla_{\BU}\log \bar{\rho}(\PSU_n(\BA) | \PSU_{0}(\BA)) \overset{\text{d}}{=} -\frac{\epsilon_n}{\sqrt{1-\gamma_n}}\,,
\]
which motivates, in order to match the noise $\epsilon_n$, the definition
\[
\epsilon_{\theta}(n, \BU, \BA) := -\sqrt{1 - \gamma_n} \SC_{\theta}(n \Delta t, \BU, \BA)
\]
for $n=1, \ldots, N$, $\BU \in \mathbb{R}^{d_{\mathcal{U}}}$ and $\BA \in \mathbb{R}^{d_{\mathcal{A}}}$.
Finally, applying the reparameterization trick and ignoring the prefactor $1/(1 - \gamma_n)$ leads to the following objective function.
\begin{definition}
The discrete denoising score matching (DDSM) objective is defined as
\begin{equation}\label{eqn:DSM}
\mathcal{L}_{\mathrm{DDSM}}(\theta) := \frac{1}{N}\sum_{n=1}^N\mathbb{E}_{\nu_{\BA}} \mathbb{E}_{\bar{\rho}(\BU_0|\BA)} \mathbb{E}_{\epsilon_n} \left[\| \epsilon_{\theta}(n, \sqrt{\gamma_n}\PSU_0(\BA) + \sqrt{1-\gamma_n}\epsilon_n, \BA) - \epsilon_n \|^2\right],
\end{equation}
where 
  $\mathbb{E}_{\epsilon_n}$ is the expectation taken with respect to the standard normal distribution and $\mathbb{E}_{\bar{\rho}(\BU_0|\BA)}$ is with respect to the data distribution $p(\BU_0|\BA)$ from the operator learning task.    
\end{definition}
\begin{remark}
    It should be emphasized that the index for $\epsilon_{\theta}$ starts at $n = 1$, which corresponds to learning $\Psi(\Delta t, \BU, \BA)$. 
    In other words, the DDPM model is not required to learn the score function at $t=0$, i.e., $\Psi(0, \BU, \BA)$, which is not be well-defined when $U(0; \BA)$ is Dirac.
\end{remark}

The training stage of DDPM involves minimizing the above objective, learning a noise predictor $\epsilon_{\theta}$ that approximates Gaussian noise $\epsilon_n$ with neural networks at each time step $n$.
In fact, it can be shown that minimizing the DDSM objective~\eqref{eqn:DSM} is equivalent to minimizing the negative variational lower bound (NVLB) of the NLL loss~\eqref{eqn:NLL-loss} (see Appendix~\ref{app:NVLB} for the derivation), which justifies the application of DDPM under the proposed probabilistic operator learning framework.     

Once the score function is learned, the inference stage of DDPM involves discretizing 
the reverse SDE~\eqref{eqn:reverse-SDE}, thereby
providing a generative model to transport the standard Gaussian $\mathcal{N}(0, \mathbf{I})$ to the desired conditional distribution $\bar{\rho}_0(\cdot | \BA)$.
To this end, we consider the following weak first-order discretization to $\PSV((N-n+1)\Delta t; \BA))$,
\[
\PSV_{n-1}(\BA) = \frac{1}{\sqrt{1-\beta_n}} \left(\PSV_{n}(\BA) + \beta_n \SC_{\theta}(n \Delta t, \PSV_{n}(\BA), \BA)\right) + \sqrt{\sigma_n} \epsilon_n, \qquad \PSV_{N}(\BA) \sim \rho_{N\Delta t}(\cdot | \BA)
\]
for $n = N, N-1, \ldots, 1$, where $\sigma_n$ is a sequence that satisfies $\beta_n^{-1}\sigma_n \to 1$ as $n \to \infty$. 
In practice, $\PSV_{N}(\BA)$ is sampled from $\mathcal{N}(0, \mathbf{I})$ for $N$ sufficiently large.
Note that $\PSV_{n}(\BA)$ is a first-order approximation to the continuous time reverse process $\PSV((N-n)\Delta t; \BA))$ in the weak sense and hence it approximates the forward process $\PSU_n(\BA)$ weakly. Moreover, the Markov chain $V_{0:N}(\BA)$ depends on $\theta$ since the drift term depends on $\Psi_{\theta}$.

We denote $\bar{\rho}_{\theta}(\BV_1, \ldots, \BV_N | \BA)$ the conditional path density of $\PSV_{0:N}(\BA)$, i.e., 
\[
\bar{\rho}_{\theta}(\BV_0,\BV_1,\dots,\BV_N | \BA)d\BV_0 \dots d\BV_N = \mathbb{P}_{\BA}(\PSV_0(\BA) \in d\BV_0, \ldots, \PSV_N(\BA) \in d\BV_N).
\]
The corresponding reverse transition density for $\PSV_{n}(\BA)$ to $\PSV_{n-1}(\BA)$ is given by
\begin{equation}\label{eqn:reverse-transition}
    \bar{\rho}_{\theta}(\BV_{n-1} | \BV_n, \BA) := \mathcal{N}( m_{\theta}(n, \BV_n, \BA), \Sigma_{\theta}(n, \BV_n, \BA)), 
\end{equation}
where  
\[
m_{\theta}(n, \BV_n, \BA)
=
\frac{1}{\sqrt{1-\beta_n}} \left(\BV_{n} - \frac{\beta_n}{\sqrt{1 - \gamma_n}} \epsilon_{\theta}(n, \BV_{n}, \BA) \right)
\]
and 
\[
\Sigma_{\theta}(n, \BV_n, \BA) = \sigma_n \mathbf{I}.
\]
Since $\PSV_n(\BA)$ is only required to be a weak first-order approximation (in $\Delta t$) to $\PSV(n\Delta t; \BA)$,
we have some flexibility to choose the variance $\sigma_n$.
For the two particular cases in the operator learning task, that is, the Gaussian noise setting and the noise-free setting, 
the following theorem provides a guide for the choice of $\sigma_n$, \cite{sohl2015deep}.
\begin{theorem}\label{thm:optimal-sigma}
Let $\theta^*$ be the solution to the DDSM objective~\eqref{eqn:DSM}.
Choose
\begin{equation}\label{eqn:cov}
\sigma_n
= \begin{cases}
\dfrac{1 - \gamma_{n-1}}{1 - \gamma_n}\beta_n & 
       \mbox{noise-free,}\\[8pt]
\beta_n  & \mbox{standard Gaussian}
         \end{cases}
\end{equation}
for the variance of the reverse Markov chain $\PSV_{0:N}(\BA)$.
Then the full-path likelihood of the Markov chain $\PSV_{0:N}(\BA)$ coincides with that of the Markov chain $\PSU_{0:N}(\BA)$, i.e., 
\[
\mathbb{E}_{\nu_{\BA}}\mathbb{E}_{\bar{\rho}_{\theta^*}(\cdot | \BA)}\left[-\log {\bar{\rho}}_{\theta^*}(\PSV_{0:N}(\BA) | \BA) \right]
=
\mathbb{E}_{\nu_{\BA}}\mathbb{E}_{\bar{\rho}(\cdot | \BA)}\left[-\log \bar{\rho}(\PSU_{0:N}(\BA) | \BA) \right]\,,
\]
where $\mathbb{E}_{\bar{\rho}(\cdot | \BA)}$ is the expectation taken with respect to all possible realizations of the path $\PSU_{0:N}(\BA)$.
\end{theorem}
The result essentially states that the parameterization~\eqref{eqn:cov} is optimal in the sense that the full-path likelihood under 
the approximated density $\bar{\rho}_{\theta^*}$ recovers the true full-path likelihood.
In Appendix~\ref{app:cov-choice}, we prove a stronger version of the above result. As we shall see in Section~\ref{sec:experiments}, numerical experiments justify that the choice of $\sigma_n$ has an important impact on the inference stage. 
In the case of a data set with non-Gaussian noise, 
a linear combination of the above two covariances can be used by
introducing an additional hyperparameter. 
Note that the covariance $\Sigma_{\theta}(n, \PSV_n(\BA), \BA)$ in~\eqref{eqn:cov} is not parameterized by $\theta$ and is independent of $\PSV_n(\BA)$ and $\BA$.

The resulting PDNO algorithm consists of two stages: 
the training stage that minimizes $\mathcal{L}_{\mathrm{DDSM}}(\theta)$ to learn the reversed conditional Markov model $\bar{\rho}_{\theta}(\BV_{n-1}| \BV_n,\BA)$ and the inference stage that samples sequentially from $\bar{\rho}_{\theta}(\BV_{n-1}| \BV_n,\BA)$, as depicted in the schematic diagram~\ref{fig:schemeDDMP}.
Algorithm~\ref{alg:PDNO} summarizes the procedure for training and sampling of PDNO. 
As indicated in~\eqref{eqn:reverse-transition}, we parameterize the noise predictor $\epsilon_{\theta}$ for $\bar{\rho}_{\theta}(\BV_{n-1}| \BV_n,\BA)$, which has been empirically observed to produce higher precision than parameterizing the mean $m_{\theta}$ directly.
For implementation, we follow the architecture presented in~\cite{ho2020denoising} to use a UNet backbone to represent the noise predictor $\epsilon_{\theta}(n, \BV_n, \BA)$ with the sinusoidal position embedding for $n$, \cite{UNET}. 
In comparison with the unconditional DDPM, the conditional DDPM utilized by PDNO takes an additional argument $\BA$ as the input to the UNet. In our implementation, $\BA$ and $\BV_n$ are channel-wise concatenated before being fed into the UNet.

\begin{algorithm}
\caption{Probabilistic diffusion neural operator (PDNO)}\label{alg:PDNO}
\begin{algorithmic}[1]
\Algphase{Training Stage:} 
\State given the training data set $\mathcal{D}_{\eta}$ of size ${M_{\text{train}}}$
\State \textbf{repeat}
\State sample $(\BA, \BU^{\eta})$ from $\mathcal{D}_{\eta}$
\State sample $n \sim \text{Unif}({1:N})$
\State sample $\epsilon \sim \mathcal{N}(0, \mathbf{I})$
\State stochastic gradient descent on $\mathcal{L}_{\text{DDSM}}(\theta)$ in~\eqref{eqn:DSM} 
\State \textbf{until} converged to obtain $\epsilon_{\theta}$ 
\Algphase{Inference Stage:}
\State given the test data set $\widetilde{\mathcal{D}}$ of size ${M_{\text{test}}}$
\For{test input $\BA$ in $\widetilde{\mathcal{D}}$}
\For{$m = 1, \cdots, M_{\text{infer}}$}
\State $\BV_N^{(m)}(\BA) \sim \mathcal{N}(0, \mathbf{I})$
\For{$n = N, \cdots, 1$}
\State $\BV_{n-1}^{(m)}(\BA) \sim \rho_{\theta}(\BV_{n-1}|\BV_n^{(m)}, \BA)$ in~\eqref{eqn:reverse-transition}
\EndFor
\EndFor
\State \textbf{return} statistics of $\{\BV_0^{(m)}(\BA)\}_{m=1}^{M_{\text{infer}}}$
\EndFor
\end{algorithmic}
\end{algorithm}

As a probabilistic model, PDNO learns a class of conditional distributions $q_{\theta}(\cdot|\BA)$ rather than the deterministic solution operator $s_{\theta}$. 
For a given test input $\BA$, $M_{\text{infer}}$ samples generated from $q_{\theta}(\cdot|\BA)$ are required to obtain various statistics of the predicted output.
However, it should be emphasized that, in the case of noise-free data where we learn a class of point delta distribution $\delta_{\BU}(\cdot|\BA)$ concentrated at $\BU$, only a single sample is required to estimate the corresponding solution provided that the model is well trained.
In practice, we validate the learned model $\delta_{\BU}(\cdot|\BA)$ with a diagnostic run on the test data set to check whether the estimated standard deviation is approximately zero.

\begin{figure}[ht]
    \centering
    \includegraphics[width=0.9\textwidth]{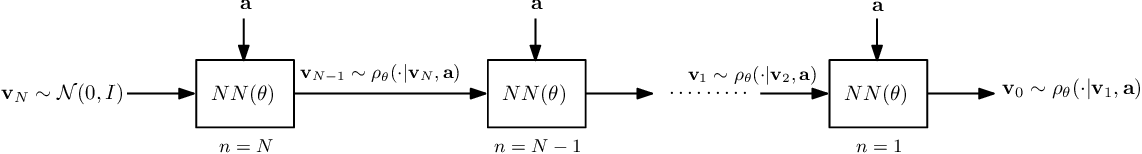}
    \caption{A schematic diagram of the inference stage of PDNO} 
    \label{fig:schemeDDMP}
\end{figure}

\subsection{Extension to functional spaces}\label{sec:ddpm-functional}
As we shall demonstrate in Section~\ref{sec:experiments}, the proposed PDNO method, built on a finite-dimensional DDPM, achieves remarkable performance across a variety of numerical benchmarks. Nevertheless, the method fails to be discretization invariant since the underlying DDPM sampler is intrinsically modeled by a finite-dimensional SDE. Recent studies on DDPM in functional spaces have opened the possibility of extending our PDNO method to varying grid resolutions~\cite{pidstrigach2023infinite, baldassari2023conditional, lim2023score, lim2023score_hilbert, franzese2025generative}. Although this work primarily focuses on learning the finite-dimensional operator $s_{\theta}$ via DDPM in finite-dimensional spaces, we briefly discuss extensions to functional spaces for two reasons: (i) the Hilbert space formulation of DDPM subsumes the finite-dimensional DDPM as a special instance, and (ii) the nonexistence of the density in the noise-free setting can be addressed through a measure-theoretic treatment in infinite dimensions, thereby providing a rigorous interpretation of noise-free operator learning (see Remark~\ref{remark:noise-free-OL}).

To extend our framework to functional spaces, we assume the perturbed model
\[u^{\eta} = \mathcal{S}_{\theta}(a) + \eta,\]
where $\eta \sim \mathcal{N}(0, \lambda \mathcal{C})$ is a $\mathcal{U}$-valued Gaussian random field with a scaled covariance operator $\lambda\mathcal{C}$ for a constant $\lambda \geq 0$.
In other words, for each input function $a \in \mathcal{A}$, we postulate that the output $u \in \mathcal{U}$ is perturbed by a Gaussian random field noise $\eta$ and hence $u^{\eta}$ is a Gaussian random field $\mathcal{N}(\mathcal{S}^\dagger(a), \lambda \mathcal{C})$.
The objective is to learn a generative model that generates samples from $\mathcal{N}(\mathcal{S}^\dagger(a), \lambda \mathcal{C})$ for each input $a$. As in the finite-dimensional setting, we consider both the noisy case $\lambda > 0$ and the noise-free case $\lambda \to 0$. 

We assume that both the input and output spaces are separable Hilbert spaces, denoted by $(\mathcal{A}, \langle \cdot, \cdot \rangle_{\mathcal{A}})$ and $(\mathcal{U}, \langle \cdot, \cdot \rangle_{\mathcal{U}})$, respectively.
For each random field input realization $a \in \mathcal{A}$, we denote $(\Omega, \mathcal{F}, \mathbb{P}_a)$ the complete conditional probability space associated with $a$.
For the sake of clarity and simplicity, we consider a class of $\mathcal{U}$-valued time-homogeneous Ornstein-Uhlenbeck process $\{U(t; a)\}_{t \in [0, T]}$ parameterized by $a$,  
\begin{equation}\label{eqn:forward-sde-infinite}
\begin{split}
&d U(t; a) = -\frac{1}{2}\PSU(t; a) \, dt + dW^{\mathcal{U}}(t)\,,\\
&U(0; a) \sim \mathcal{N}(\mathcal{S}^{\dagger}(a), \lambda \mathcal{C})\,
\end{split}
\end{equation}
for $t \in [0, T]$, 
where $W^{\mathcal{U}}(t)$ is a $\mathcal{U}$-valued $\mathcal{C}$-Wiener process with the covariance operator $\mathcal{C}$ (see Appendix~\ref{app:Wiener-process}).
We use the notation $\mu_t(\cdot|a)$ for the induced measure of $U(t; a)$ in the infinite-dimensional setting, i.e., 
\[\mu_t(B | a) = \mathbb{P}_a(U(t; a) \in B), \qquad B \in \sigma(\mathcal{U}).\]
Note that by definition $\mu_0(\cdot|a) = \mathcal{N}(\mathcal{S}^{\dagger}(a), \lambda \mathcal{C})$.
Furthermore, for each $t$, we denote by $\mu_{[0, t]}(\cdot|a)$ the path-space measure induced by the process $\{U(s; a)\}_{s \in [0, t]}$.

From a theoretical perspective, the score function~\eqref{eqn:score-function} is no longer well-defined since there is no Lebesgue measure in infinite-dimensional spaces.
To avoid dealing with the nonexistence of the score function in an infinite-dimensional setting, we follow the approach developed in~\cite{pidstrigach2023infinite}. 
Motivated by~\eqref{eqn:OU-score-function} that expresses the score function as a conditional expectation, we define the score operator in Hilbert spaces as follows. 
\begin{definition}
The score operator $\Psi: [0, T] \times \mathcal{U} \times \mathcal{A} \to \mathcal{U}$ associated with the SDE~\eqref{eqn:forward-sde-infinite} is defined as
\begin{equation*}
\begin{split}
\Psi(t, u, a) 
&:= -\frac{1}{1 - \text{e}^{-t}} \mathbb{E}_{\mu_{[0, T]}(\cdot|a)}\left[ U(t; a) -\text{e}^{-\frac{t}{2}} U(0; a)  ~|~ U(t; a) = u\right]\,,
\end{split}
\end{equation*}
where $\mathbb{E}_{\mu_{[0, T]}(\cdot|a)}$ 
is the expectation taken with respect to all paths of the OU process $\{U(t; a)\}_{t \in [0, T]}$
conditioned on the input $a \in \mathcal{A}$.     
\end{definition}
To time-reverse the dynamics we define the reverse process $\{V(t; a)\}_{t \in [0, T]}$ by $V(t; a) := U(T - t; a)$. 
The following theorem specifies the governing SDE dynamics of $V(t; a)$.
\begin{theorem}
For any $a \in \mathcal{A}$, assume that
    \[\sup_{t \in [0, T]} \mathbb{E}_{\mu_t(\cdot|a)}\left[\|\Psi(t, U(t; a), a)\|_{\mathcal{U}}^2\right] < \infty\,.\]
    Then $V(t; a)$ is the solution to the following reverse time SDE 
\begin{equation}\label{eqn:reverse-SDE-infinite}
\begin{split}
&d\PSV(t; a) = \left(\frac{1}{2} \PSV(t; a) + \SC(T-t, \PSV(t; a),  a)\right)\, dt + d \widetilde{W}^{\mathcal{U}}(t)\,,\\
        &\PSV(0;a) \sim \mu_T(\cdot | a)\,
\end{split}
\end{equation}
for $t \in [0, T]$,
where $\widetilde{W}^{\mathcal{U}}(t)$ is a $\mathcal{U}$-valued $\mathcal{C}$-Wiener process adapted to the reversed filtration $\{\widetilde{\mathcal{F}}_t\}$.    
\end{theorem}
The proof of the theorem relies on applying the finite-dimensional time reversal results to a projected finite-dimensional subspace and showing that the finite-dimensional dynamics converges to the correct infinite-dimensional dynamics. We refer the reader to~\cite{baldassari2023conditional, baldassari2024taming, pidstrigach2023infinite} for the proof.  
\begin{remark}\label{remark:noise-free-OL}
Now we can give a precise meaning to noise-free operator learning.
For the noise-free case, i.e., $\lambda = 0$, the score operator reduces to
\[\Psi(t, u, a)  = -\frac{1}{1 - \text{e}^{-t}} \left( u -\text{e}^{-\frac{t}{2}} \mathcal{S}^\dagger(a)  \right)\,.\] 
In this case, it can be readily shown that 
\begin{equation}\label{eqn:noise-free-estimates}
\mathbb{E}_{\mu_t(\cdot|a)}\left[ \|\Psi(t, U(t; a), a)\|_{\mathcal{U}}^2 \right] = \frac{1}{1 - \text{e}^{-t}}\text{Tr}(\mathcal{C}), \quad \quad t > 0\,, 
\end{equation}
which converges to infinity as $t \to 0$. 
This suggests that in the noise-free case, the reverse SDE may not be well-defined as $t \to 0$.
However, notice that the theorem holds true on intervals $[t_{\text{start}}, T]$ for any $t_{\text{start}} > 0$. In the noise-free case, the forward SDE transports the true output $U(0; a) = \mathcal{S}^\dagger(a)$ to $U(t_{\text{start}}; a) \approx \mathcal{S}^\dagger(a)$ by a small noise for $t_{\text{start}}$ sufficiently small. Therefore, in practice, we can run DDPM only on the interval $[t_{\text{start}}, T]$ and learn a generative model for sampling $U(t_{\text{start}}; a)$ instead. 
In other words, in the noise-free case, the DDPM model is equivalent to learning a Gaussian distribution with a small standard deviation. 
\end{remark}
To learn the score operator, we can parameterize the operator $\Psi$ by $\Psi_{\theta}$ and optimize the following score matching objective,
\[
\mathcal{L}_{\text{SM}}(\theta) = \frac{1}{T} \int_0^T \ell_{\mathrm{SM}}(t;\theta) \, dt\,,
\]
where
\[
\ell_{\mathrm{SM}}(t;\theta):= \mathbb{E}_{\nu_a}\mathbb{E}_{\mu_t(\cdot|a)}\left[ \|\Psi_{\theta}(t, U(t; a), a) - \Psi(t, U(t; a), a)\|_{\mathcal{U}}^2\right]\,.
\]
In practice, the objective is replaced by a tractable denoising score matching objective.
\begin{theorem}\label{thm:DSM-objective-continuous-infinite}
Suppose that the true score function satisfies that
\[
\sup_{t \in (0, T]}\mathbb{E}_{\nu_{a}}\mathbb{E}_{\mu_{t}(\cdot |a)} \left[\|\SC(t, \PSU(t; a), a)\|_{\mathcal{U}}^2\right] < \infty\,.
\]
Then minimizing $\mathcal{L}_{\text{SM}}(\theta)$ is equivalent to minimizing the following DSM objective,
\begin{equation}\label{eqn:DSM-continuous-infinite}
\mathcal{L}_{\text{DSM}}(\theta) := \frac{1}{T} \int_0^T \ell_{\text{DSM}}(t; \theta) \, dt
\end{equation}
with
\[
\ell_{\text{DSM}}(t; \theta)
:=
\mathbb{E}_{\nu_a}
\mathbb{E}_{\mu_{0, t}(\cdot | a)}
\left[\left\|\SC_{\theta}(t, \PSU(t; a), a) + \frac{1}{1 - \text{e}^{-t}}(U(t; a) -\text{e}^{-\frac{t}{2}} U(0; a)) \right\|_{\mathcal{U}}^2\right]\,,
\]
where 
$\mathbb{E}_{\mu_{0, t}(\cdot | a)}$ is the expectation taken with respect to the joint $(U(0; a), U(t; a))$ conditioned on the input $a \in \mathcal{A}$. 
\end{theorem}
The proof is presented in Appendix~\ref{app:DSM-objective-continuous-infinite} for completeness.
Owing to~\eqref{eqn:noise-free-estimates}, the integrability condition in the above theorem fails for the noise-free case. However, a closer inspection of the proof in Appendix~\ref{app:DSM-objective-continuous-infinite} immediately implies the following result for the noise-free case.
\begin{corollary}\label{cor:DSM-objective-continuous-infinite-Dirac}
When $U(0; a) \sim \delta_{\mathcal{S}^\dagger(a)}(\cdot | a)$ for each $a \in \mathcal{A}$, the DSM objective~\eqref{eqn:DSM-continuous-infinite} reduces to 
\[
\frac{1}{T}\int_0^T
\mathbb{E}_{\nu_a}
\mathbb{E}_{\mu_{t|0}(\cdot | \mathcal{S}^{\dagger}(a))} \left[\left\|\SC_{\theta}(t, \PSU(t; a), a) + \frac{1}{1 - \text{e}^{-t}}(U(t; a) -\text{e}^{-\frac{t}{2}}  \mathcal{S}^{\dagger}(a))\right\|_{\mathcal{U}}^2\right]\,dt,
\]
where $\mathbb{E}_{\mu_{t|0}(\cdot | \mathcal{S}^{\dagger}(a))}$ is the expectation taken with respect to the conditional distribution of $U(t; a)$ given $U(0; a) = \mathcal{S}^{\dagger}(a)$. 
\end{corollary}

To train the model, we discretize the forward SDE~\eqref{eqn:forward-sde-infinite} to obtain the forward Markov chain $\{U_{0:N}(a)\}$ and train the score operator $\Psi_{\theta}$ by minimizing the discrete objective 
$\mathcal{L}_{\mathrm{DDSM}}(\theta)$ given by
\[
\frac{1}{N}\sum_{n=1}^N\mathbb{E}_{\nu_{a}}
\mathbb{E}_{\bar{\mu}_{0, n}(\cdot | a)}
\left[
\left\|
\SC_{\theta}(n, \PSU_n(a), a) + \frac{1}{1 - \text{e}^{-n\Delta t}}(U_n(a) -\text{e}^{-\frac{n \Delta t}{2}} U_0(a))
\right\|_{\mathcal{U}}^2\,,
\right] 
\]
where $\mathbb{E}_{\bar{\mu}_{0, n}(\cdot | a)}$ denotes the expectation taken with respect to the joint distribution of $(U_0(a), U_n(a)) $ conditioned on $a \in \mathcal{A}$.
For inference from the model, a discretization of the reverse time process~\eqref{eqn:reverse-SDE-infinite} leads to a generative model for sampling from the desired distribution $\mathcal{N}(\mathcal{S}^\dagger(a), \lambda \mathcal{C})$. 

The functional space formulation establishes the theoretical foundation for the discretization invariant property of DDPM. For practical implementation, however, suitable neural network architectures are needed to approximate the score-matching operator $\Psi$ in infinite-dimensional domains. Unfortunately, the default UNet architecture used in DDPM is not resolution-independent. Neural operators, such as FNO and ONet, are powerful tools for approximating operators in functional spaces, \cite{FNO, ONET}. Yet, our numerical experiments with FNO did not yield accurate simulation results, suggesting that the UNet architecture may play a critical role in the success of DDPM. As an alternative, attention-based architectures combined with UNet could potentially enable resolution invariance during training, which we leave for future work.

%
%

\section{Computational benchmarks}\label{sec:experiments}
We demonstrate the efficacy of PDNO on several benchmarks and
provide a comparison with the results obtained by FNO~\cite{FNO}. 
Denote $\widetilde{D}$ as the test data set.
We evaluate the prediction accuracy and uncertainty of the model $\sopa$ 
by the mean relative $L_2$-error (MRLE) and the mean standard deviation (MSTD) that are defined by
\begin{equation*}
\frac{1}{|\mathcal{\widetilde{D}}|} \sum_{(\BA, \BU) \in \mathcal{\widetilde{D}}} \frac{\|\mbox{mean}(\{\BV_0^{(m)}(\BA)\}_{m=1}^{M_{\text{infer}}}) - \BU \|_2}{\|\BU\|_2}\, \quad\mbox{and}\quad 
\frac{1}{|\mathcal{\widetilde{D}}|} \sum_{(\BA, \BU) \in \mathcal{\widetilde{D}}}
\frac{\|\mbox{std}(\{\BV_0^{(m)}(\BA) \}_{m=1}^{M_{\text{infer}}}) \|_2}{M_{\text{infer}}}\,,
\end{equation*}
respectively. 
All experiments are implemented with a single Tesla T4 GPU.
More implementation details about the experiments are given in Appendix~\ref{sec:detailed-experiments}.

\subsection{Elliptic equation in 1D}
We start with a problem
whose analytical solution is known explicitly:
\begin{equation}\label{eqn:1d-problem}
\begin{split}
-(a(x) u^{\prime}(x))^{\prime} &= 0\,, \qquad x \in (-1, 1)\,,
\end{split}
\end{equation}
with the boundary condition $u(-1) = 0$ and $u(+1) = 1$, and
$u^\prime$ denotes $du/dx$. We are interested in learning the 
solution operator $\mathcal{S}^{\dagger}$ mapping a log-normal field $a \in \mathcal{A} = L_2((-1, 1); \mathbb{R})$ to the solution field $u \in \mathcal{U} =  L_2((-1, 1); \mathbb{R})$. 
The random input field in this example corresponds to
the log-normal field 
$a(x) = \mathrm{e}^{\beta V(x)}$ with $\beta = 0.1$, 
where the potential $V$ is parameterized by a random vector
$(A_1, \ldots, A_n, B_1, \ldots, B_n)$, namely,
$$
V(x) = \frac{1}{\sqrt{n}} \sum_{k=1}^{n} A_k \cos(\pi k x) + B_k \sin(\pi k x)\,,
$$ 
where $A_k$ and $B_k$
are independent standard normal random variables. 
In this example, we choose $n = 5$ and thus the log-normal random field is parameterized by $10$ independent Gaussians. 
One can easily verify
that $V(x)$ is a Gaussian random field with zero mean and the
covariance kernel
\[
\mathrm{cov}(x_1, x_2) := \mathbb{E}[V(x_1) V(x_2)] =  n^{-1}\sum_{k=1}^{n} \cos(\pi k (x_2 - x_1))\,.
\]
Moreover, by direct calculation, it can be easily verified that the exact solution of the
problem~\eqref{eqn:1d-problem} is
\[
u(x) = \left(\int_{-1}^{1} a^{-1}(\xi)\, d\xi\right)^{-1} \int_{-1}^{x} a^{-1}(\xi)\, d\xi, \qquad x \in (-1, 1).
\]

We test the performance of PDNO on a Gaussian noise data set. 
The noisy data set $\mathcal{D}_{\eta} = \{(\BA^{(n)}, \BU^{\eta, (n)})\}_{n=1}^{M_{\text{train}}}$ with $M_{\text{train}} = 10^4$ is generated by corrupting the output $\BU$ with a Gaussian noise $\eta \sim \mathcal{N}(0, \sigma^2 \mathbf{I})$, i.e., $\BU^{\eta} = \BU + \eta$.
We repeat the data generation process for $\sigma = 1\%, \sigma = 5\%, \sigma = 100\%$.
We aim to learn the mapping $s_{\theta}: \mathbb{R}^{128} \to \mathbb{R}^{128}$ as well as the standard deviation of the noise $\eta$.
The set of hyper-parameters used for training the PDNO is summarized in Table~\ref{tab:params-test1}.
In Figure~\ref{fig:1d_noisy}, we show the learned solution $u(\cdot)$ for a given input realization $a$ in the setting of $\sigma = 1\%$ (left) and $\sigma = 5\%$ (right).
The result demonstrates that PDNO is able to learn both the mean and the standard deviation from a noisy data set, which is a key advantage of PDNO over other deterministic methods.

\begin{figure}
    \centering
    \includegraphics[width=0.9\textwidth]{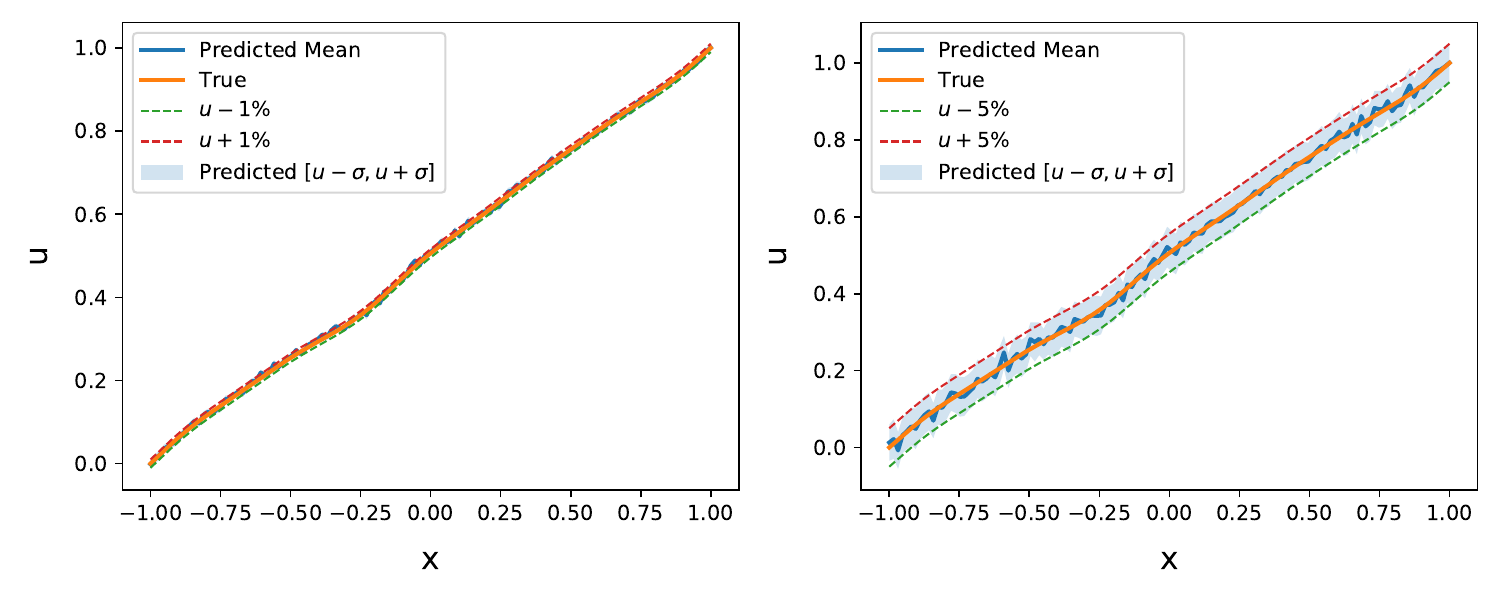}
    \caption{Solutions to the 1D elliptic equation learned by PDNO.
    The model is trained on a data set corrupted by an additive Gaussian noise $\mathcal{N}(0,\sigma^2 \mathrm{I})$.
    {\bf Left}: the predicted mean and confidence interval of the solution $u$ corresponds to a particular input $a$. The model is trained on the noisy data set with $\sigma = 1\%$. 
    {\bf Right}: the predicted mean and confidence interval of the solution $u$ corresponds to a particular input $a$. The model is trained on the noisy data set with $\sigma = 5\%$.
    The predicted confidence intervals (blue shaded band) learned by PDNO coincide with the theoretical confidence intervals (lower bound and upper bound indicated by green dashed line and red dashed line, respectively).}
    \label{fig:1d_noisy}
\end{figure}

Finally, we train our model on a highly noisy data set with $\sigma = 100\%$. Figure~\ref{fig:1d_large_noise} presents the learned solution and the corresponding learned standard deviation for two input realizations $\BA$ from the test set. To illustrate the typical magnitude of noise affecting the training data, we artificially add standard Gaussian noise to the true solution $\BU$. The results clearly indicate that PDNO successfully learns the solution even under such extreme noise levels.


\begin{figure}
    \centering
    \includegraphics[width=0.8\linewidth]{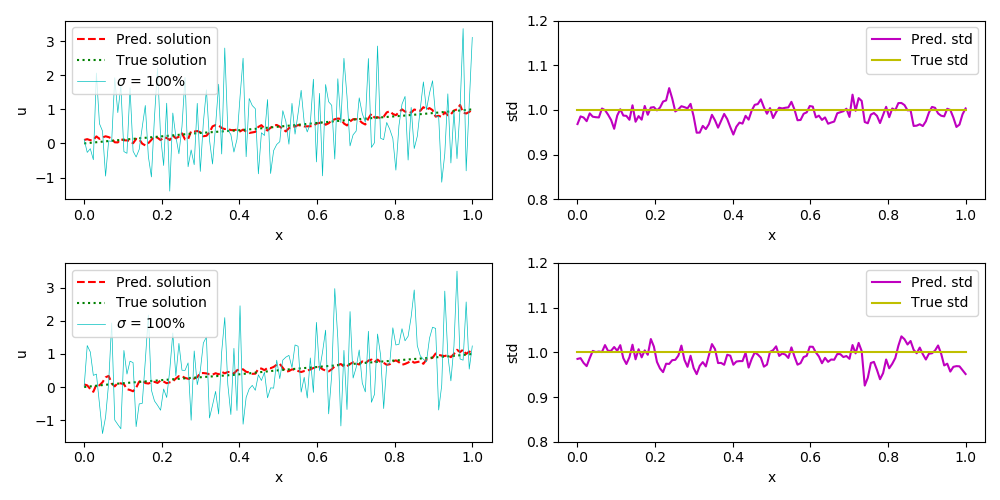}
    \caption{Solution to the 1D elliptic equation trained on a data set corrupted by an additive Gaussian noise $\mathcal{N}(0, \mathrm{I})$.
    Both the predicted solutions (left) and the standard deviations (right) for two given realizations of the input $a$ are shown. A standard Gaussian noise ($\sigma = 100\%$) is added to the true solution $\BU$ to demonstrate the magnitude of the perturbation. }
    \label{fig:1d_large_noise}
\end{figure}

\subsection{Elliptic equation in 2D}
To test the method on higher-dimensional input-output data sets, we apply the proposed 
PDNO to learn a solution operator 
\[\SOP:a \in L^{\infty}(D; \mathbb{R})\mapsto u \in H_0^1(D; \mathbb{R})\] to
a 2D elliptic PDE:
\begin{equation}\label{eqn:2d-elliptic}
\begin{split}
-\nabla \cdot (a(x) \nabla u(x)) &= f(x)\, \qquad x \in D\\
u(x) &= 0\, \qquad\;\;\;\;\; x \in \partial D
\end{split}
\end{equation}
where the random input $a$ is
obtained from another Gaussian random field $\tilde a$ by
\[
a(x) = \alpha + (\beta-\alpha) \IND{\{\tilde{a}(x)\geq 0\}}, \qquad
\tilde a \sim \mathcal{N}(0, (-\Delta + \tau^2)^{-r}).
\]
See Appendix~\ref{app:Wiener-process} for the precise meaning. 
The Gaussian field $\tilde a$ has mean zero and covariance operator 
$\mathcal{C} = (-\Delta + \tau^2)^{-r}$, where $\Delta$ is the Laplace operator
with Neumann boundary conditions on $D$. Note that $\tau$ defines
the inverse correlation length and $r$ controls the regularity of the field.
In correspondence to experiments reported in \cite{FNO} we have chosen $\tau = 3$, $r=2$, and $\alpha=4$, $\beta=12$. 
The right-hand side is $f(x)\equiv 1$.
All input-output data pairs $(a, u)$ in this example are represented
by their grid values $(\BA, \BU)$ on the mesh of size $d_{\mathcal{U}}$, thus $\BA, \BU \in \mathbb{R}^{d_{\mathcal{U}}}$.

\subsubsection{Noise-free data set}
We use a data set $\mathcal{D}$ of size $M_{\text{train}} = 1000$ for training.
The size of
meshes used in the experiments and the corresponding prediction results are reported in Table~\ref{tab:darcy_error}.
The set of parameters used for training the PDNO is summarized in Table~\ref{tab:params-test2}.



As shown in Table~\ref{tab:darcy_error}, in the noise-free data setting, PDNO achieves lower MRLE
than FNO for various mesh sizes. 


\begin{figure}[!h]
    \centering
    \includegraphics[width=1.0\textwidth]{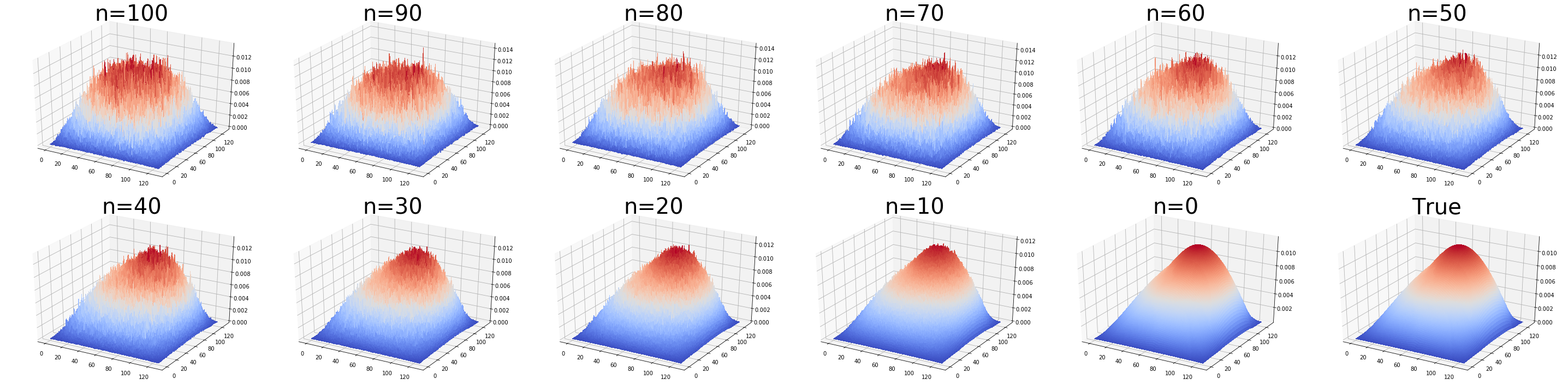}
    \caption{The reverse time evolution of PDNO applied to the elliptic equation in 2D for a given input field $a$. Starting from a standard 2D Gaussian sample at $n  = N = 100$, the plot shows the intermediate output samples after every $10$ steps. The final sample at $n = 0$ corresponds to the solution learned by PDNO.}
    \label{fig:darcyflow_history}
\end{figure}

\begin{table}[!h]
\caption{MRLE and MSTD for the 2D elliptic PDE in Test 2.}
\begin{center}
  \begin{tabular}{llcccccc}
    \toprule
    \multirow{2}{*}{Method} &
    \multirow{2}{*}{Parameters} &
      \multicolumn{2}{c}{$d_{\mathcal{U}}=64^2$} &
      \multicolumn{2}{c}{$d_{\mathcal{U}}=128^2$} &
      \multicolumn{2}{c}{$d_{\mathcal{U}}=256^2$} \\
      & & {MRLE} & {MSTD} & {MRLE} & {MSTD} & {MRLE} & {MSTD} \\
      \midrule
    FNO & 2, 376, 449 & 0.0060 & N/A & 0.0078 & N/A & 0.0050 & N/A \\
    PDNO &  35, 707, 841 & 0.0046 & $5.08$e$-07$ & 0.0022 & $3.52$e$-07$ & 0.0037 & $1.21$e$-05$ \\
    \bottomrule
  \end{tabular}
\end{center}
\label{tab:darcy_error}
\end{table}
\subsubsection{Noisy data set} 
In this example, we present a study of learning the solution operator along with an uncertainty estimate
from a noisy data set $\mathcal{D}_{\eta}$ in which the output data points $\BU$ are 
corrupted by a $\mathbb{R}^{d_{\mathcal{U}}}$-variate Gaussian mixture noise
\[
\eta \sim \alpha \sum_{i=1}^{n_g} \pi_i \mathcal{N}(\mu_i, \sigma_i^2 \mathbf{I})\,, 
\]
where the constant $\alpha = 0.2 \, u_\mathrm{max}$ is chosen to control the relative error of the perturbation. The maximum value of the exact solution $u_\mathrm{max} = 0.01$ over the prescribed mesh 
is the maximum value of the true solution $u$ over the prescribed mesh. 
Parameters used in the experiment are summarized in Table~\ref{tab:noisydata}. The Gaussian mixture noise $\eta$ leads to an $8\%$ and $16\%$ perturbation to the true output $\BU$ in case 1 and case 2, respectively.

\begin{table}[!h]
\caption{Operator learning from the noisy data set: 2D elliptic problem. The training data are perturbed by
a mean-zero additive noise defined by a Gaussian mixture.
The relative error refers to fluctuations at $u_{\max}$.}
\begin{center}
  \begin{tabular}{cccc}
    \toprule
    Mixture &\multirow{2}{*}{Case 1}  & \multirow{2}{*}{Case 2} \\
    parameters &                      &          \\
     \midrule
    $\pi_i$ & $[1/3,1/3,1/3]$ & $[1/4,1/4,1/4,1/4]$ \\
    $\mu_i$ & $[-0.5,0,0.5]$  & $[-1, -0.5, 0.5, 1]$ \\
    $\sigma_i$ & $[0.01,0.01,0.01]$ & $[0.01,0.1,0.2,0.05]$ \\
    \midrule
    Standard dev. & $8.167 \times 10^{-4}$ & $1.598\times 10^{-3}$ \\
    Rel. error & $\approx 8\%$ & $\approx 16\%$  \\
    \bottomrule
  \end{tabular}
\end{center}
\label{tab:noisydata}
\end{table}

 \begin{figure}
 \centering
 \begin{subfigure}[b]{0.48\textwidth}
 \centering
    \includegraphics[width=\textwidth]{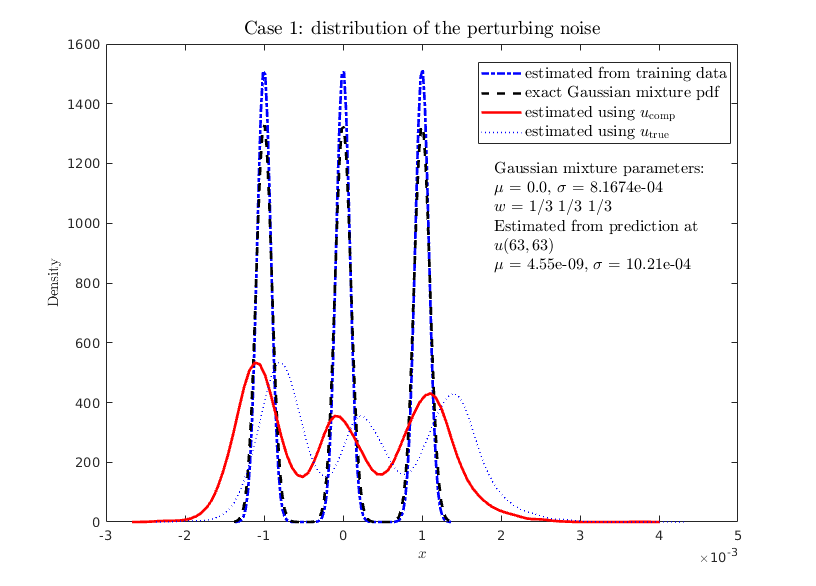}
    \captionsetup{font=small} 
        \caption{Case 1}
 \end{subfigure}
 \begin{subfigure}[b]{0.48\textwidth}
 \centering
    \includegraphics[width=\textwidth]{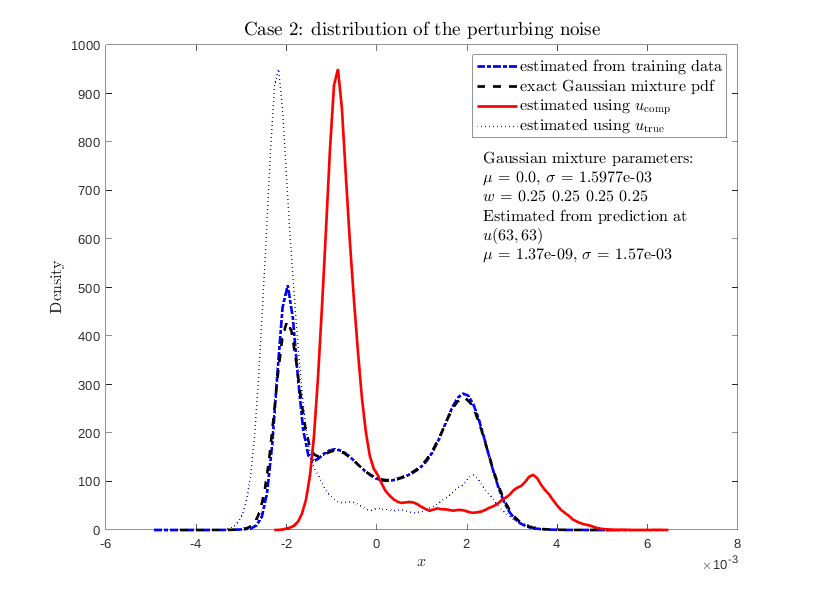}
    \captionsetup{font=small} 
        \caption{Case 2}
 \end{subfigure}

 \caption{Kernel density estimation of the distribution of the noise for the studied cases.}
 \label{fig:noise_distribution}
 \end{figure}

For a given input $a$, our simulation aims to recover the noise distribution of $\eta$ at $a$. For inference, we generate $M_{\text{infer}} = 10^4$ samples from the learned DDPM model and use the kernel density estimation (KDE) to obtain the estimated density function. 
We present the results in Figure~\ref{fig:noise_distribution} where we recover the noise $\eta$ at the center of the $128 \times 128$ grid (i.e., index $(63, 63)$). In Case 1 where the output is corrupted by a mixture of three Gaussians, our PDNO is able to recover the 3-modal pattern of the mixture distribution. In Case 2, with a more challenging noise perturbation, PDNO learned the correct first two moments of $\eta$. However, it failed to fully capture the multi-modality pattern.


\subsection{Burgers' equation}
We consider Burgers' equation with periodic boundary conditions 
on the spatial domain $D=(0,1)$,
\begin{equation}\label{eqn:burgers}
\begin{split}
\partial_t u(x, t) + u(x, t)\partial_x u(x, t) &= 
\nu \partial_{xx} u(x, t), \qquad x \in D\,,\;\;\; t \in (0, 1]\\
u(x, 0) &= u_0(x), \qquad x \in D
\end{split}
\end{equation}
where $\nu$ is the viscosity coefficient, which is set to $1.0$.
In this experiment, the random input $a$ is the initial condition $u_0 \in L_{\mathrm{per}}^2(D; \R)$ that is chosen to be a zero mean Gaussian random field with the covariance operator $625(-\Delta +25 I)^{-2}$ with Neumann boundary conditions. 
Similar to~\cite{FNO}, we aim to learn the operator 
$\SOP: L_{\mathrm{per}}^2(D; \R) \to L_{\mathrm{per}}^p(D; \R)$, for $p>1$,
mapping the initial condition $u_0$ to the solution $u(\cdot,1)$. 
We use the same data set
$\mathcal{D}$ provided in \cite{FNO}
in order to make a performance comparison of our PDNO with the FNO. The set of hyper-parameters used for training PDNO is summarized in Table~\ref{tab:params-test3}.


The prediction result is summarized in Table~\ref{tab:burgers} and the decay of 
MRLE is shown in Figure~\ref{fig:error_vs_epoch}.
We observed that FNO achieves the claimed mesh invariant property in this example. Although PDNO achieves low MRLE for the experiment, FNO performs significantly better than PDNO on various mesh sizes. 
We argue that the great performance of FNO for this example is due to the periodic boundary conditions used in the equation. 


\begin{table}[!h]
\caption{MRLE and MSTD on Burgers' equation in Test $3$.}
\begin{center}
  \begin{tabular}{llcccccc}
    \toprule
    \multirow{2}{*}{Method} &
    \multirow{2}{*}{Parameters} &
      \multicolumn{2}{c}{$d_{\mathcal{U}}$=512} &
      \multicolumn{2}{c}{$d_{\mathcal{U}}$=1024} &
      \multicolumn{2}{c}{$d_{\mathcal{U}}$=2048} \\
      & & {MRLE} & {MSTD} & {MRLE} & {MSTD} & {MRLE} & {MSTD} \\
      \midrule
    FNO & $582, 849$ &  $0.0005$ & N/A & $0.0005$ & N/A & $0.0005$ & N/A \\
    PDNO & $3, 980, 129$ & $0.0047$ & $5.47$e$-04$ & $0.0053$ & $7.93$e$-04$ & $0.0152$ & $1.51$e$-03$ \\
    \bottomrule
  \end{tabular}
\end{center}
\label{tab:burgers}
\end{table}

\subsection{Advection equation}
This example aims to investigate the performance 
of PDNO
for an advection PDE whose solution is highly non-smooth.
We test our method for the 1D advection equation with the inflow condition 
on the domain $D=(0,1)$:
\begin{equation}
    \begin{split}
        &\partial_t u(x,t) +  \partial_x u(x,t) = 0   \qquad x \in D\,,\\
        &u(x,0) = u_0(x)\,,\;\;\;\mbox{and \;\;$u(0,t) = 0$ \qquad $t\geq 0$.}
    \end{split}
\end{equation}
The solution operator learned in this problem is the mapping from the initial condition $u_0$ to a solution at a fixed time $t_f>0$, i.e., 
\[\SOP:u_0 \in L^{\infty}(D; \mathbb{R})\mapsto u(\cdot,t_f) \in L^{\infty}(D; \mathbb{R}).\]
The random input field $u_0$ of the data set is identified with the initial condition
\[
u_0(x) = \IND{\{\tilde{a}(x) \geq 0\}}\,,
\]
where $\tilde{a}(\cdot)$ is a Gaussian random field of mean zero and
with the covariance kernel $(-\partial_{xx} + \tau^2)^{-r}$
\[
\tilde{a} \sim \mathcal{N}(0, (-\partial_{xx} + \tau^2)^{-r}).
\]
Here, the Laplace operator
$\partial_{xx}$ is with the zero Neumann boundary conditions.
For comparison with the FNO method and results in \cite{de2022cost}, we use the same values $\tau = 100$ and $r=2$. 
The output field of the data set is the solution $u(x,t_f)$.
All input-output pairs in this example are represented
by their grid values on the mesh of the size $d_{\mathcal{U}}$.
The set of parameters used for training PDNO for this example is summarized in Table~\ref{tab:params-test4}.

\subsubsection{Noise-free data set}
For the noise-free setting, we generate $N = 1500$ input-output training pairs 
$\mathcal{D} = \{(\BA^{(n)}, \BU^{(n)})\}_{n=1}^{M_{\text{train}}}$ with the mesh size $d_{\mathcal{U}} = 1024$ and sub-sample to obtain the data at a lower resolution.
We benchmark PDNO with the FNO developed in~\cite{FNO} on a test data set $\widetilde{\mathcal{D}}$ of $M_{\text{test}} = 300$ input-output pairs.
The size of
meshes used in the experiments and the corresponding prediction results are reported in Table~\ref{tab:adv_error}.
In this highly discontinuous example, PDNO significantly outperforms FNO in all mesh sizes.
For the mesh size $d_{\mathcal{U}}=1024$, 
both the PDNO and FNO solutions at two particular inputs $a$ from $\widetilde{\mathcal{D}}$ are presented in~Figure~\ref{fig:advection}.
The decay of MRLE with respect to the number of training epochs for the advection problem is shown in Figure~\ref{fig:error_vs_epoch} (right).
Compared to deterministic methods such as FNO, the experiment demonstrates that PDNO is able to learn 
very fine details of the highly non-smooth solution. 
Finally, Figure~\ref{fig:advection_history} demonstrates how PDNO transports a Gaussian sample to the output solution $u$ corresponding to a given input $a$ in $N = 100$ steps.

\begin{figure}[!h]
    \centering
    \includegraphics[width=1.0\textwidth]{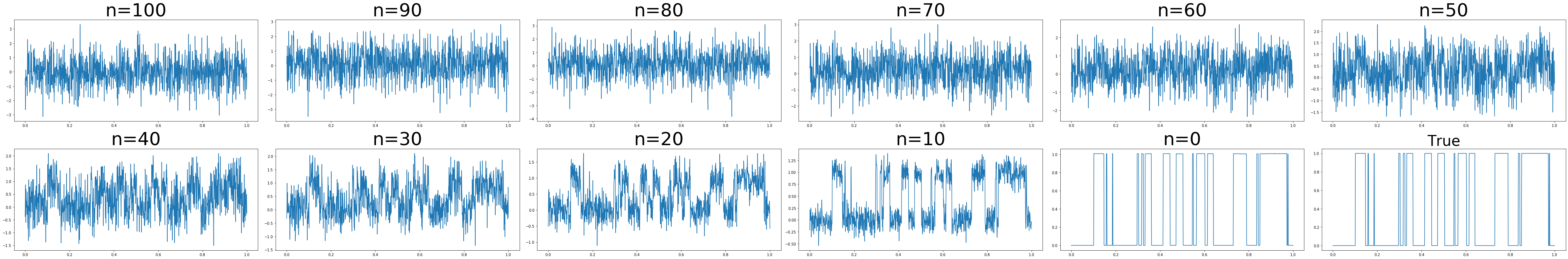}
    \caption{The reverse time evolution of PDNO applied to the advection problem for a given input field $a$. Starting from a standard Gaussian sample at $n = N = 100$, the plot shows the intermediate output samples after every $10$ steps. The final sample at $n = 0$ corresponds to the solution learned by PDNO.}
    \label{fig:advection_history}
\end{figure}

\begin{table}[!h]
\caption{MRLE and MSTD for the Advection problem in Test 4.}
\begin{center}
  \begin{tabular}{llcccccc}
    \toprule
    \multirow{2}{*}{Method} &
    \multirow{2}{*}{Parameters} &
      \multicolumn{2}{c}{$d_{\mathcal{U}}$=256} &
      \multicolumn{2}{c}{$d_{\mathcal{U}}$=512} &
      \multicolumn{2}{c}{$d_{\mathcal{U}}$=1024} \\
      & & {MRLE} & {MSTD} & {MRLE} & {MSTD} & {MRLE} & {MSTD} \\
      \midrule
    FNO & $2, 328, 961$ &  $0.0514$ & N/A & $0.1072$ & N/A & $0.1325$ & N/A \\
    PDNO & $3, 980, 129$ & $0.0062$ & $7.91$e$-04$ & $0.0081$ & $7.38$e$-04$ & $0.0445$ & $2.59$e$-03$ \\
    \bottomrule
  \end{tabular}
\end{center}
\label{tab:adv_error}
\end{table}


\begin{figure}[!h]
    \centering
    \includegraphics[width=1.0\textwidth]{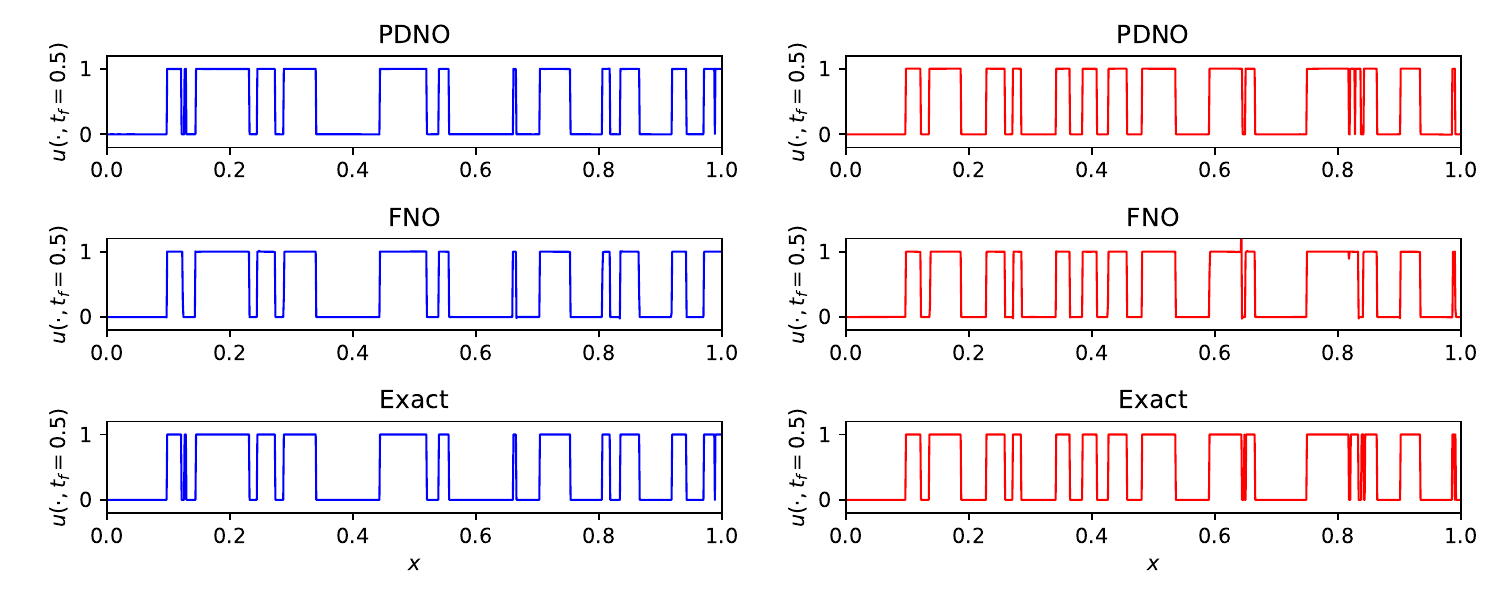}
    \caption{Solutions to the Advection problem at $t_f = 0.1$, i.e., $u(\cdot, 0.1)$ predicted by learned models for particular initial condition inputs $u_0$ from the test data set. 
    {\bf Left}: 
    solution $u(\cdot, 0.5)$ corresponds to $u_0$ for which PDNO solution exhibits the lowest MRLE.
    {\bf Right}: solution $u(\cdot, 0.5)$ corresponds to $u_0$ for which PDNO solution exhibits the largest MRLE.}
    \label{fig:advection}
\end{figure}



\subsubsection{Noisy case}
For the noisy setting, we synthesize a noisy training data set $\mathcal{D}_{\eta} = \{(\BA^{(n)}, \BU^{\eta, (n)})\}_{n=1}^{M_{\text{train}}}$ by adding a Gaussian noise $\eta \sim \mathcal{N}(0, \sigma^2 \mathbf{I})$ to the output $\BU$, i.e., $\BU^{\eta} = \BU + \eta$.
To demonstrate that PDNO is capable of learning from a highly noisy dataset, we set $\sigma = 50\%$, which completely destroys the pattern of the true output solution. Two typical input-output samples in $\mathcal{D}_{\eta}$ are shown in Figure~\ref{fig:noisy_advection_sample}, which shows that the outputs are highly noisy. In Figure~\ref{fig:learned_solutions_advection}, the solutions learned by PDNO are compared with the noise-free outputs corresponding to the input-output pairs shown in Figure~\ref{fig:noisy_advection_sample}.  
It can be seen that PDNO manages to recover the $0-1$ pattern of the solution to the advection equation, and the prediction inaccuracy is primarily due to the Gibbs phenomenon. 
The number of PDNO ensembles we use to generate these predicted solutions is
$M_{\text{infer}} = 10^4$ and the noise can be further reduced by increasing $M_{\text{infer}}$.

\begin{figure}
    \centering
    \includegraphics[width=1.0\linewidth]{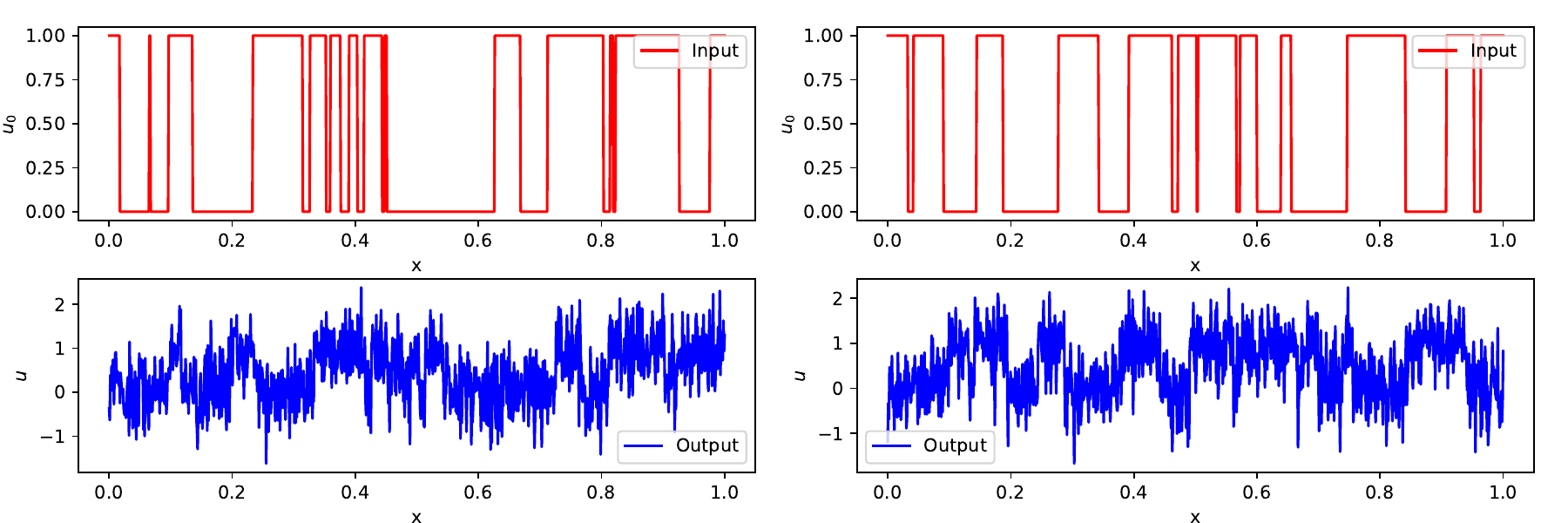}
    \caption{Two typical input-output samples from the noisy data set for training. Top: the input initial condition $u_0$; Bottom: the corresponding noisy output solution $u$ corrupted by $\eta \sim \mathcal{N}(0, \sigma^2 \mathbf{I})$ with $\sigma = 50\%$.}
    \label{fig:noisy_advection_sample}
\end{figure}

\begin{figure}
    \centering
    \includegraphics[width=1.0\linewidth]{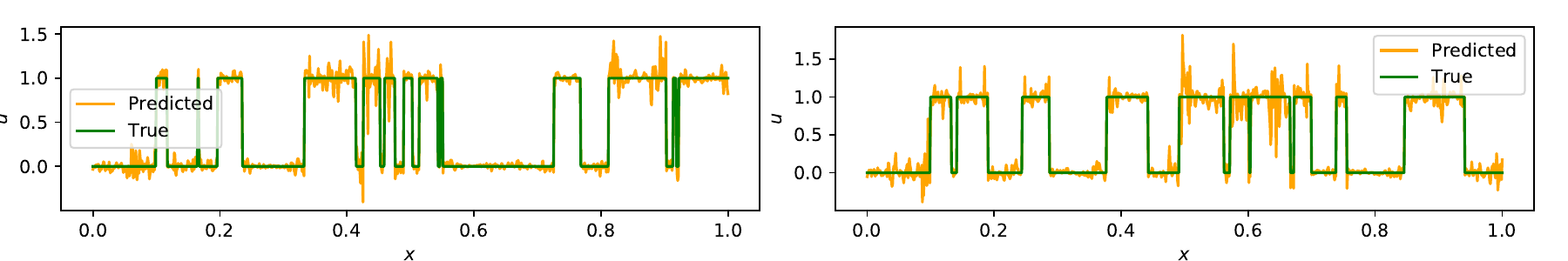}
    \caption{The two learned output solutions corresponding to the inputs shown in Figure~\ref{fig:noisy_advection_sample}.}
    \label{fig:learned_solutions_advection}
\end{figure}


\begin{figure}
    \centering
    \includegraphics[width=0.95\textwidth]{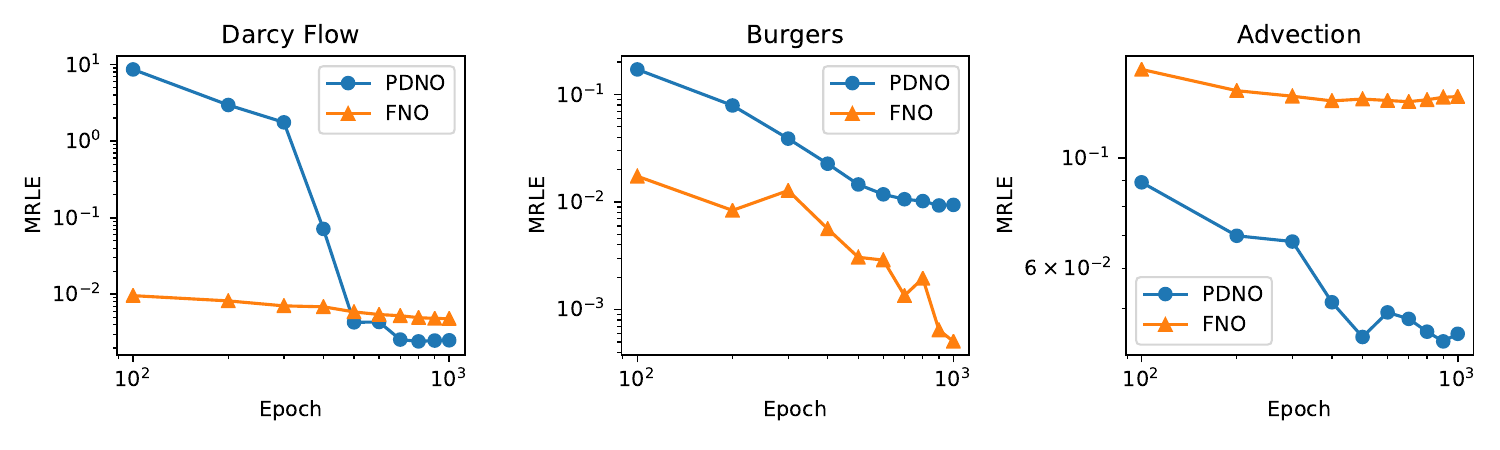}
    \caption{Decay of the MRLE with respect to the number of training epochs.
    {\bf Left:} 2D elliptic PDE with $d_{\mathcal{U}} = 128$. 
    {\bf Center:} Burgers' equation with $d_{\mathcal{U}} = 1024$.
    {\bf Right:}  Advection PDE with $d_{\mathcal{U}} = 1024$.}
    \label{fig:error_vs_epoch}
\end{figure}

\section*{Conclusion}\label{sec:conclusion}
We have presented a probabilistic generative model based on DDPM for learning the solution operator of parametric PDEs. 
Our method differs from other state-of-the-art approaches in that we view the input-to-output mapping as a family of conditional distributions conditioned on input parameters. 
Therefore, our model predicts the solution with quantified uncertainty. While the method has various advantages, we also observed several limitations in numerical experiments. 

More specifically, the current method learns a model that is trained on a predefined discretization mesh, which limits the ability to predict solutions at points outside the mesh. More importantly, the mesh-dependent property requires us to retrain the model in order to predict over different resolution meshes.
Furthermore, since sampling of $u|a$ requires sequential sampling for $N$ steps, the inference of our model is relatively slow compared to other deterministic models.

To deal with the mesh-dependent issue, one promising approach is to directly formulate the diffusion model over functional spaces. 
Alternatively, it is possible to carry out diffusion learning over an appropriate latent space~\cite{rombach2022high}, e.g., over the Fourier space~\cite{FNO} or the spectral space~\cite{phillips2022spectral}.
Finally, we can take advantage of the cascaded diffusion model pipeline to make predictions on a fine mesh based on a model pre-trained on a coarse mesh~\cite{ho2022cascaded}.
As for the sampling efficiency, the denoising diffusion implicit models can be applied to accelerate the inference with a slight compromise in accuracy~\cite{song2020denoising}.
Finally, the current implementation utilizes a UNet architecture to represent the transition density, which often involves a larger number of parameters than the architecture for FNO. 
A more efficient DDPM-based solver
calls for a special-purpose architecture for the operator learning application.

On the other hand, the probabilistic framework is directly applicable to learning solution operators from {\it noisy data sets}. The sampling nature of the solution prediction then allows for characterizing the uncertainty due to model training on noise-corrupted data. In numerical benchmarks we were able to partially recover corrupting noise distribution and/or its moments.  

We evaluated the computational performance of the proposed PDNO method against Fourier Neural Operators. Studied numerical benchmarks
showed that PDNO achieves comparable accuracy while additionally recovering the noise magnitude in data sets corrupted by additive noise.

\bibliographystyle{amsplain}
\bibliography{diffusion_model}

%
%
\appendix

\numberwithin{equation}{section}
\makeatletter 
\newcommand{\section@cntformat}{Appendix \thesection:\ }
\makeatother

\setcounter{table}{0}
\renewcommand\thetable{A\arabic{table}}

\setcounter{figure}{0}
\renewcommand\thefigure{A\arabic{figure}}



%
%
\section{Proof of Theorem~\ref{thm:DSM-objective-continuous}}\label{app:remainder-R}
Write the SM objective~\eqref{eqn:SM-objective} (before averaging with respect to $\BA$ and time) as
\begin{equation}\label{eqn:app-SM}
\begin{split}
&\mathbb{E}_{\rho_t(\BU | \BA)} \left[\|\SC_{\theta}(t, \PSU(t; \BA), \BA) - \SC(t, \PSU(t; \BA), \BA)\|^2\right]\\
=&
\mathbb{E}_{\rho_t(\BU | \BA)} 
\left[
\|\Psi_{\theta}(t, \PSU(t; \BA), \BA)\|^2
\right]
+ 
\mathbb{E}_{\rho_t(\BU | \BA)} 
\left[
\|\Psi(t, \PSU(t; \BA), \BA)\|^2
\right]\\
&~~~~-
2\mathbb{E}_{\rho_t(\BU | \BA)} 
\left[
\langle \Psi_{\theta}(t, \PSU(t; \BA), \BA), \Psi(t, \PSU(t; \BA), \BA) \rangle 
\right],
\end{split}
\end{equation}
Note that the last term can be further expanded into
\begin{equation*}
    \begin{split}
        &\mathbb{E}_{\rho_t(\BU | \BA)}\left[\left\langle \Psi_{\theta}(t, \PSU(t; \BA), \BA), \frac{\nabla_{\BU} \rho_t(U(t; \BA) | \BA)}{\rho_t(U(t; \BA) | \BA)}\right\rangle  \right]\\
        =&
        \int_{\BU} \int_{\BU_0} \left\langle \Psi_{\theta}(t, \BU, \BA), \nabla_{\BU} \log \rho_{t|0} (\BU | \BU_0)\right\rangle \rho_{t|0} (\BU | \BU_0) \rho_0 (\BU_0 | \BA)\, d\BU_0 d\BU\\
        =&
        \mathbb{E}_{\rho_0(\BU_0 | \BA)}\mathbb{E}_{\rho_{t|0}(\BU | \BU_0)}\left[\left\langle \Psi_{\theta}(t, \PSU(t; \BA), \BA), \nabla_{\BU} \log \rho_{t|0}(\PSU(t; \BA) | \PSU(0; \BA))\right\rangle  \right]\,,
    \end{split}
\end{equation*}
where 
$\rho_{t|0} (\BU | \BU_0)$\footnote{Note that when conditioned on $\PSU_0 = \BU_0$, $\rho_{t|0} (\BU | \BU_0) = \rho_{t|0} (\BU | \BU_0, \BA)$ and hence the omission of $\BA$ in the notation.} is the density of $\PSU(t; \BA)$ conditioned on the initial state $\PSU(0; \BA) = \BU_0$
and 
the expectation $\mathbb{E}_{\rho_0(\BU_0 | \BA)}$ is taken with respect to the initial density $\rho_0(\BU_0 | \BA)$.
Recall that the DSM objective~\eqref{eqn:DSM-continuous} (before averaging with respect to $\BA$ and time) is
\begin{equation}\label{eqn:app-DSM}
    \begin{split}
        &\mathbb{E}_{\rho_0(\BU_0 | \BA)}\mathbb{E}_{\rho_{t|0}(\BU|\BU_0)} \left[\|\SC_{\theta}(t, \PSU(t; \BA), \BA) -  \nabla_{\BU} \log \rho_{t|0} (\PSU(t; \BA) | \PSU(0;\BA)) \|^2\right]\\
        = & \mathbb{E}_{\rho_{t}(\BU|\BA)} \left[\|\SC_{\theta}(t, \PSU(t; \BA), \BA)\|^2\right] 
        + \mathbb{E}_{\rho_0(\BU_0 | \BA)}\mathbb{E}_{\rho_{t|0}(\BU|\BU_0)}  \left[\| \nabla_{\BU} \log \rho_{t|0} (\PSU(t; \BA) | \PSU(0;\BA)) \|^2\right] \\
        & ~~~-2\mathbb{E}_{\rho_0(\BU_0 | \BA)}\mathbb{E}_{\rho_{t|0}(\BU | \BU_0)}\left[\left\langle \Psi_{\theta}(t, \PSU(t; \BA), \BA), \nabla_{\BU} \log \rho_{t|0}(\PSU(t; \BA) | \PSU(0; \BA))\right\rangle  \right] 
    \end{split}
\end{equation}
Subtracting~\eqref{eqn:app-DSM} from~\eqref{eqn:app-SM} leads to 
\[
R(t; \BA) := \mathbb{E}_{\rho_t(\BU | \BA)} 
\left[\|\Psi(t, \PSU(t; \BA), \BA)\|^2 \right] - \mathbb{E}_{\rho_0(\BU_0 | \BA)}\mathbb{E}_{\rho_{t|0}(\BU|\BU_0)}  \left[\| \nabla_{\BU}\log\rho_{t|0} (\PSU(t; \BA) | \PSU(0;\BA)) \|^2\right],
\]
which is independent of the parameter $\theta$. 

It is trivial to verify that $R(0; \BA) = 0$ for any $\BA \in \mathbb{R}^{d_{\mathcal{A}}}$. 
It remains to show that $R(t; \BA)$ is finite for any $t \in (0, T]$ and $\BA \in \mathbb{R}^{d_{\mathcal{A}}}$.
For $t \in (0, T]$, through a random-time-change $t \to \bar{\alpha}(t) := \int_0^t \alpha(s) \, ds$ for the time-homogeneous OU process~\eqref{eqn:OU-process} (when $C$ is identity), we derive the conditional density for the forward SDE~\eqref{eqn:forward-sde-finite},
\[
\rho_{t|0} (\BU | \BU_0) = \frac{1}{\sqrt{(2 \pi (1 - \text{e}^{-\bar{\alpha}(t)}))^{d_{\mathcal{U}}}}}  \text{exp}\left({-\frac{1}{2(1 - \text{e}^{-\bar{\alpha}(t)})} \|\BU - \text{e}^{-\bar{\alpha}(t)/2}\BU_0\|^2} \right).
\]
Consequently, for all $\BA \in \mathbb{R}^{d_{\mathcal{U}}}$ and all $t \in (0, T]$,
\begin{equation}\label{eqn:conditional-score-estimate}
\mathbb{E}_{\rho_0(\BU_0 | \BA)}\mathbb{E}_{\rho_{t|0}(\BU|\BU_0)}  \left[\| \nabla_{\BU}\log\rho_{t|0} (\PSU(t; \BA) | \PSU(0;\BA)) \|^2\right] 
=
d_{\mathcal{U}}\left(1 - \text{e}^{-\bar{\alpha}(t)}\right)^{-1} < \infty\,,
\end{equation}
where the term $d_{\mathcal{U}}$ is the trace of the identity covariance mapping from $\mathbb{R}^{d_{\mathcal{U}}}$ to $\mathbb{R}^{d_{\mathcal{U}}}$. 
Thus, under the assumption of Theorem~\ref{thm:DSM-objective-continuous}, DSM objective~\eqref{eqn:DSM-continuous} and SM objective~\eqref{eqn:SM-objective} have the same minimizer for any $t \geq 0$.  
\begin{remark}
    In the infinite-dimensional setting where $d_\mathcal{U} \to \infty$, \eqref{eqn:conditional-score-estimate} converges to infinity due to the identity diffusion term involved in the forward SDE~\eqref{eqn:forward-sde-finite}. This suggests the necessity of replacing the identity operator with a trace-class operator $\mathcal{C}$ in the diffusion term when $d_\mathcal{U} \to \infty$. 
\end{remark}

%
%

\section{Proof of~\texorpdfstring{\eqref{eqn:OU-score-function}}{(OU-score-function)}}\label{app:OU-score-function}
The proof is the same as that for the unconditional setting presented in~\cite{pidstrigach2023infinite}, which relies on the Gaussianity of the conditional density
\[
\rho_{t|0}(\BU | \BU_0) \propto \text{exp}\left({-\frac{1}{2(1 - \text{e}^{-t})} (\BU - \text{e}^{-t/2}\BU_0)^{\text{T}}  C^{-1} (\BU - \text{e}^{-t/2}\BU_0)} \right), \qquad t > 0\,.
\]
Then the score function can be written as 
\begin{equation*}
\begin{split}
\nabla_{\BU} \log \rho_t(\BU | \BA) &= \nabla_{\BU} \log \int \rho_{t|0}(\BU |\BU_0) \rho_0(\BU_0 | \BA) \, d\BU_0\\
& = -\frac{1}{1 - \text{e}^{-t}} \int C^{-1}(\BU - \text{e}^{-t/2}\BU_0) \frac{\rho_{t|0}(\BU | \BU_0)}{\rho_t(\BU | \BA)} \rho_0(\BU_0|\BA) \, d\BU_0\\
& = -\frac{1}{1 - \text{e}^{-t}} \int C^{-1}(\BU - \text{e}^{-t/2}\BU_0) \rho_{0|t}(\BU_0 | \BU, \BA) \, d\BU_0\\
& = -\frac{1}{1 - \text{e}^{-t}} C^{-1} \mathbb{E}_{\BA} \left[U(t; \BA) - \text{e}^{-t/2}U(0; \BA) | U(t; \BA) = \BU \right],
\end{split}
\end{equation*}
where $\rho_{0|t}(\BU_0 | \BU, \BA)$ is the conditional density of $(U(0; \BA) | U(t; \BA) =\BU, A = \BA)$ and here we have used $A$ to denote the random variable that generates a realization $\BA$.

%
%
\section{Proof of Theorem~\ref{thm:DSM-objective-continuous-infinite}}\label{app:DSM-objective-continuous-infinite}
The proof is similar to the finite-dimensional case. 
First, note that the SM objective in the functional space setting is 
\begin{equation}\label{eqn:app-fSM}
\begin{split}
 &\mathbb{E}_{\mu_t(\cdot | a)} 
\left[\|\Psi_{\theta}(t, \PSU(t; a), a) - \Psi(t, \PSU(t; a), a)\|_{\mathcal{U}}^2 \right]\\
= & \mathbb{E}_{\mu_t(\cdot | a)} \left[\|\Psi_{\theta}(t, \PSU(t; a), a)\|_{\mathcal{U}}^2 \right]
+ \mathbb{E}_{\mu_t(\cdot | a)} \left[\|\Psi(t, \PSU(t; a), a)\|_{\mathcal{U}}^2 \right]\\
\qquad \qquad & -2 \mathbb{E}_{\mu_t(\cdot | a)} \left[ \langle \Psi_{\theta}(t, \PSU(t; a), a), \Psi(t, \PSU(t; a), a) \rangle_{\mathcal{U}}\right]\,.
\end{split}
\end{equation}
The last cross term can be written as 
\begin{equation*}
\begin{split}
(1-\text{e}^{-t})^{-1}\mathbb{E}_{\mu_t(\cdot | a)} \left[ \mathbb{E}_{\mu_{0|t}(\cdot | U(t;a), a)}\left[ \langle \Psi_{\theta}(t, \PSU(t; a), a),  U(t;a) -  \text{e}^{-t/2} U(0; a) \rangle_{\mathcal{U}}   \right] \right]\\
=
(1-\text{e}^{-t})^{-1}\mathbb{E}_{\mu_0(\cdot | a)} \left[ \mathbb{E}_{\mu_{t|0}(\cdot | U(0;a), a)}\left[ \langle \Psi_{\theta}(t, \PSU(t; a), a),  U(t;a) -  \text{e}^{-t/2} U(0; a) \rangle_{\mathcal{U}}   \right] \right]\,
\end{split}
\end{equation*}
Also, note that the DSM objective is
\begin{equation}\label{eqn:app-fDSM}
    \begin{split}
&\mathbb{E}_{\mu_{0, t}(\cdot | a)}
\left[\left\|\SC_{\theta}(t, \PSU(t; a), a) + (1 - \text{e}^{-t})^{-1}(U(t; a) -\text{e}^{-\frac{t}{2}} U(0; a)) \right\|_{\mathcal{U}}^2\right]\\
= &  \mathbb{E}_{\mu_t(\cdot | a)} \left[\|\Psi_{\theta}(t, \PSU(t; a), a)\|_{\mathcal{U}}^2 \right]
+ \mathbb{E}_{\mu_{0, t}(\cdot | a)}
\left[\left\|(1 - \text{e}^{-t})^{-1}(U(t; a) -\text{e}^{-\frac{t}{2}} U(0; a)) \right\|_{\mathcal{U}}^2\right]\\
\qquad \qquad & -2 (1-\text{e}^{-t})^{-1}\mathbb{E}_{\mu_0(\cdot | a)} \left[ \mathbb{E}_{\mu_{t|0}(\cdot | U(0;a), a)}\left[ \langle \Psi_{\theta}(t, \PSU(t; a), a),  U(t;a) -  \text{e}^{-t/2} U(0; a) \rangle_{\mathcal{U}}   \right] \right]\,.
    \end{split}
\end{equation}
Comparing~\eqref{eqn:app-fSM} with~\eqref{eqn:app-fDSM} leads to the difference 
\begin{equation*}
    \begin{split}
R(t; a) &:= \mathbb{E}_{\mu_t(\cdot | a)} 
\left[\|\Psi(t, \PSU(t; a), a)\|_{\mathcal{U}}^2 \right] \\
&\quad - \mathbb{E}_{\mu_0(\cdot | a)}\mathbb{E}_{\mu_{t|0}(\cdot|U(0; a))}  \left[\| (1-\text{e}^{-t})^{-1}(U(t; a) - \text{e}^{-t/2} U(0; a)) \|_{\mathcal{U}}^2\right]\,.
\end{split}
\end{equation*}
Clearly, $R(0; a) = 0$ for any $a$. 
Note that for any $t > 0$, we have
\[
\mathbb{E}_{\mu_0(\cdot | a)} \mathbb{E}_{\mu_{t|0}(\cdot|U(0; a), a)}  \left[\| (1-\text{e}^{-t})^{-1}(U(t; a) - \text{e}^{-t/2} U(0; a)) \|_{\mathcal{U}}^2\right] = \frac{1}{1 - \text{e}^{-t}} \text{Tr}(\mathcal{C}) < \infty\,.
\] 
Hence, the SM objective is equivalent to the DSM objective, provided that 
\[
\sup_{t \in (0, T]}\mathbb{E}_{\nu_{a}}\mathbb{E}_{\mu_{t}(\cdot |a)} \left[\|\SC(t, \PSU(t; a), a)\|_{\mathcal{U}}^2\right] < \infty\,.
\]

To see the corollary~\ref{cor:DSM-objective-continuous-infinite-Dirac}, note that by~\eqref{eqn:noise-free-estimates}, $R(t; \BA) = 0$ for any $t \geq 0$ and $a \in \mathcal{A}$. The equivalence between SM and DSM objectives holds true trivially.


%
%

\section{DSM objective as the NVLB of the NLL}\label{app:NVLB}

Following~\cite{ho2020denoising}, we re-derive the DSM objective~\eqref{eqn:DSM} as the NVLB of the NLL loss defined in~\eqref{eqn:NLL-loss} for operator learning.
Given an input $\BA$, DDPM starts with the forward transition density (as defined in~\eqref{eqn:forward-transition})
\[
    \bar{\rho}(\BU_n | \BU_{n-1}, \BA) = \mathcal{N}(\sqrt{1 - \beta_n} \BU_{n-1}, \beta_n \mathbf{I}), \qquad \bar{\rho}(\BU_0 | \BA) = p(\BU_0 | \BA)
\]
for $n = 1, \ldots, N$
and hence the $n$-step transition density (conditioned on $\PSU_{0}(\BA) = \BU_0$)
\[
\bar{\rho}(\BU_n | \BU_{0}, \BA) = \mathcal{N}(\sqrt{\gamma_n}\BU_0, (1 - \gamma_n)\mathbf{I}),
\]
where $\gamma_n = \prod_{i=1}^n (1 - \beta_n)$.

Generally, the true reverse transition density $\bar{\rho}(\BU_{n-1} | \BU_{n})$ is intractable since the initial density $\rho(\BU_0 | \BA)$ is unknown.
However, the conditional reverse transition density $\bar{\rho}(\BU_{n-1} | \BU_{n},  \BU_{0}, \BA)$ can be derived through the Bayes' formula,
\begin{equation}\label{eqn:true-reverse-density}
\begin{split}
\bar{\rho}(\BU_{n-1} | \BU_{n},  \BU_{0}) 
= 
\bar{\rho}(\BU_{n} | \BU_{n-1})
\frac{\bar{\rho}(\BU_{n-1} | \BU_{0})}{\bar{\rho}(\BU_{n} | \BU_{0})},
\end{split}
\end{equation}
which turns out to be a Gaussian $\mathcal{N}(m(n, \BU_{n}, \BU_{0}), \Sigma(n, \BU_{n}, \BU_{0}))$
with
\begin{equation*}
m(n, \BU_{n}, \BU_{0}) = \frac{\sqrt{1 - \beta_n}(1 - \gamma_{n-1})}{1 - \gamma_n}\BU_{n} + \frac{\sqrt{{\gamma}_{n-1}}\beta_n}{1 - {\gamma}_n}\BU_{0}\,, 
\end{equation*}
and 
\begin{equation}\label{eqn:true-reverse-var}
\Sigma(n, \BU_{n}, \BU_{0}) = \frac{1 - \gamma_{n-1}}{1 - \gamma_n}\beta_n \mathbf{I}.
\end{equation}
Using the re-parameterization for $\PSU_0(\BA)$, i.e., 
\begin{equation*}
\PSU_0(\BA) \overset{\text{d}}{=} \frac{1}{\sqrt{\gamma_n}} \left(\PSU_n(\BA) - \sqrt{1 - \gamma_n} \epsilon_n \right)\,,
\end{equation*}
we can re-parameterize the mean in terms of the Gaussian noise $\epsilon_n$ as
\begin{equation}\label{eqn:true-reverse-mean}
\begin{split}
m(n, \BU_{n}, \epsilon_n) = \frac{1}{\sqrt{1 - \beta_n}} \left(\BU_{n} - \frac{\beta_n}{\sqrt{1 - \gamma_n}} \epsilon_n\right).
\end{split}
\end{equation}

The process $\PSU_{0:N}(\BA)$ can be considered as a procedure of gradually adding noise
to $\PSU_0(\BA)$ until it becomes a Gaussian noise $\PSU_N(\BA)$ for $N$ large.
Conversely, we can start from a normal distribution and gradually denoise it to obtain the desired distribution $\bar{\rho}(\BU_0 | \BA) = p(\BU_0 | \BA)$. 
This is accomplished by defining a reverse time Markov chain $\PSV_{0:N}(\BA)$ with the reverse transition density
\begin{equation}\label{eqn:approx-reverse-density}
    \bar{\rho}_{\theta}(\BV_{n-1} | \BV_n, \BA) := \mathcal{N}(m_{\theta}(n, \BV_n, \BA), \Sigma_{\theta}(n, \BV_n, \BA))
\end{equation}
and the terminal distribution
\[
\PSV_N(\BA) \sim  \mathcal{N}(0, \mathbf{I}),
\]
where $m_{\theta}(n, \BV_n, \BA)$ and $\Sigma_{\theta}(n, \BV_n, \BA)$ are the mean and the covariance, respectively. 
They are chosen to match the true mean~\eqref{eqn:true-reverse-mean} and the true covariance~\eqref{eqn:true-reverse-var}, which leads to the same parameterization as~\eqref{eqn:reverse-transition}, i.e.,
\[
m_{\theta}(n, \BV_n, \BA)
=
\frac{1}{\sqrt{1-\beta_n}} \left(\BV_{n} - \frac{\beta_n}{\sqrt{1 - \gamma_n}} \epsilon_{\theta}(n, \BV_{n}, \BA) \right)\,,
\]
where $\epsilon_{\theta}$ is a neural network that approximates Gaussian noise $\epsilon_n$.

Recall that $\bar{\rho}_{\theta}$ denotes the joint density induced by the parameterized Markov process $\PSV_{0:N}(\BA)$ and $\bar{\rho}$ denotes the joint density induced by the Markov process $\PSU_{0:N}(\BA)$.  
For notation simplicity, we omit the dependence on $\BA$ for the processes $\PSU_{0:N}(\BA)$ and $\PSV_{0:N}(\BA)$ in the following derivation.
To derive the DDPM loss function, we first rewrite the averaged (with respect to the true data distribution) log-likelihood

\begin{equation*}
\begin{split}
    &\mathbb{E}_{\bar{\rho}(\BU_0 | \BA)}[\log \bar{\rho}_{\theta}(\PSU_0 | \BA)]\\
    &=\mathbb{E}_{\bar{\rho}(\BU_0 | \BA)}\left[\log \int_{\BU_N} \ldots \int_{\BU_1} \bar{\rho}_{\theta}(\PSU_0, \BU_{1:N} | \BA) \, d\BU_{1:N}\right]\\  
    &=\mathbb{E}_{\bar{\rho}(\BU_0 | \BA)}\left[\log \int_{\BU_N} \ldots \int_{\BU_1} \frac{\bar{\rho}_{\theta}(\PSU_0, \BU_{1:N} | \BA)}{\bar{\rho}(\BU_{1:N}|\BU_0)} \bar{\rho}(\BU_{1:N}|\BU_0) \, d\BU_{1:N}\right]\\  
    &=\mathbb{E}_{\bar{\rho}(\BU_0|\BA)}\left[\log \mathbb{E}_{\bar{\rho}(\BU_{1:N}|\BU_0)}
    \left[
    \frac{\bar{\rho}_{\theta}(\PSU_{0:N} | \BA)}{\bar{\rho}(\PSU_{1:N} | \PSU_0)}
    \right]
    \right]\,.
\end{split}
\end{equation*}
Applying Jensen's inequality to the inner expectation leads to the following conditional NVLB

\begin{equation*}
    \begin{split}
        \mathbb{E}_{\bar{\rho}(\BU_0 | \BA)}[-\log \bar{\rho}_{\theta}(\PSU_0 | \BA)] \leq \mathrm{NVLB} := \mathbb{E}_{\bar{\rho}(\BU_{0:N}|\BA)}\left[\log \frac{\bar{\rho}(\PSU_{1:N}|\PSU_0)}{\bar{\rho}_{\theta}( \PSU_{0:N}|\BA)}\right]\,.
    \end{split}
\end{equation*}
Thanks to the Markov property and the Bayes' formula~\eqref{eqn:true-reverse-density},
\begin{equation*}
\begin{split}
&\text{NVLB} = \mathbb{E}_{\bar{\rho}(\BU_{0:N} | \BA)}\left[ -\log \bar{\rho}_{\theta}(\PSU_N) +\sum_{n=2}^N \log \frac{\bar{\rho}(\PSU_n |\PSU_{n-1})}{\bar{\rho}_{\theta}(\PSU_{n-1} | \PSU_{n}, \BA)} + \log \frac{\bar{\rho}(\PSU_1 |\PSU_0)}{\bar{\rho}_{\theta}(\PSU_0|\PSU_1, \BA)} \right]\\
=& \mathbb{E}_{\rho(\BU_{0:N} | \BA)}\left[\log\frac{\bar{\rho}(\PSU_N |\PSU_0 )}{\bar{\rho}_{\theta}(\PSU_N)}\right] + \sum_{n=2}^N\mathbb{E}_{\bar{\rho}(\BU_{0:N} | \BA)}\left[ \log \frac{\bar{\rho}(\PSU_{n-1} | \PSU_n, \PSU_0)}{\bar{\rho}_{\theta}(\PSU_{n-1} | \PSU_n, \BA)}\right] \\
&~~~~~~~+ \mathbb{E}_{\bar{\rho}(\BU_{0:N} | \BA)}\left[-\log \bar{\rho}_{\theta}(\PSU_0 | \PSU_1, \BA )  \right]\\
=& L_N + \sum_{n=2}^N L_{n-1} + L_0,
\end{split}
\end{equation*}
where
\begin{equation*}
    \begin{split}
        L_0(\theta; \BA) &= \mathbb{E}_{\bar{\rho}(\BU_0, \BU_1 | \BA)}\left[-\log \bar{\rho}_{\theta}(\PSU_0 | \PSU_1, \BA) \right]   \\
    L_{n-1}(\theta; \BA) &= \mathbb{E}_{\bar{\rho}(\BU_0, \BU_n | \BA)}\left[\text{KL}(\bar{\rho}(\PSU_{n-1} | \PSU_n, \PSU_0) || \bar{\rho}_{\theta}(\PSU_{n-1} | \PSU_n, \BA))\right], \qquad n = 2, \ldots, N   \\
        L_N(\theta; \BA) &= \mathbb{E}_{\bar{\rho}(\BU_0 | \BA)}\left[\text{KL}(\bar{\rho}(\PSU_N | \PSU_0) || \bar{\rho}_{\theta}(\PSU_N)) \right] 
    \end{split}
\end{equation*}

Comparing the two densities~\eqref{eqn:true-reverse-density} and~\eqref{eqn:approx-reverse-density} in the KL divergence and empirically ignoring the prefactor leads to 
\begin{equation}\label{eqn:L_t-1}
L_{n-1}(\theta; \BA) 
= \mathbb{E}_{\PSU_0 \sim p(\cdot | \BA)}\mathbb{E}_{\epsilon_n \sim \mathcal{N}(0, \mathbf{I})} \left[\|\epsilon_n - \epsilon_{\theta}(n, \sqrt{\gamma_n}\PSU_0 + \sqrt{1 - \gamma_n} \epsilon_n, \BA)\|^2\right]
\end{equation}
for $n = 2, \ldots, N$.
Similarly, we can calculate $L_0$ and ignore the prefactor, which gives 
\begin{equation}\label{eqn:L_0}
L_0(\theta; \BA) = \mathbb{E}_{\PSU_0 \sim p(\cdot | \BA)}\mathbb{E}_{\epsilon_1 \sim \mathcal{N}(0, \mathbf{I})}\left[\|\epsilon_1 - \epsilon_{\theta}(1, \sqrt{\gamma_1}\PSU_0 + \sqrt{1 - \gamma_1} \epsilon_1, \BA)\|^2\right].
\end{equation}
Unifying~\eqref{eqn:L_t-1} and~\eqref{eqn:L_0} and averaging over the input $\BA$ leads to the final loss function
\begin{equation}\label{eqn:unified_L}
    L(\theta)
= 
\frac{1}{N}\sum_{n=1}^{N}
\mathbb{E}_{\BA \sim \nu({\BA})}
\mathbb{E}_{\PSU_0 \sim p(\cdot | \BA)}\mathbb{E}_{\epsilon_n \sim \mathcal{N}(0, \mathbf{I})}\left[\|\epsilon_n - \epsilon_{\theta}\left(n, \sqrt{\gamma_n}\PSU_0 + \sqrt{1 - \gamma_n} \epsilon_n, \BA\right)\|^2\right]\,,
\end{equation}
which is consistent with the DDSM objective~\eqref{eqn:DSM}.

%
%

\section{The choice of the covariance for the reverse process}\label{app:cov-choice}
For the simplicity of notation, we prove Theorem~\ref{thm:optimal-sigma} for the unconditional setting. Indeed, we prove a stronger version of Theorem~\ref{thm:optimal-sigma} in Proposition~\ref{prop:Dirac} and Proposition~\ref{prop:Gaussian}.
We first review basic concepts in information theory.
Let $X$ and $Y$ be two random variables supported by $\mathcal{X}$ and $\mathcal{Y}$, respectively. We denote 
\[H(X) := -\int_\mathcal{X} p_X(x)\log p_X(x)\,dx\] 
the entropy of $X\sim p_X(x)$. By convention, we define the entropy of a Dirac distribution to be $-\infty$.
We recall that the conditional entropy of $Y$ given $X$ is defined as
\[
H(Y|X) := -\int_{x \in \mathcal{X}} \int_{y \in \mathcal{Y}} p_{(X, Y)}(x, y) \log p_{Y|X=x}(y|x) \, dx dy\,,
\]
where $p_{(X, Y)}$ is the joint density of $(X, Y)$ and $p_{Y|X=x}$ is the density of $Y$ conditioned on $X = x$.
Note that $H(X, Y) = H(X|Y) + H(Y)$.
Furthermore, we recall the definition of mutual information between $X$ and $Y$
\[
I(X; Y) = I(Y; X) := H_p(X) + H_p(Y) - H_p(X,Y)\, 
\]
where $H(X,Y)$ is the entropy of the joint distribution $p_{(X,Y)}(x,y)$. 

The following proposition gives the conditional entropy of the reverse transition $H(\PSU_{n-1}|\PSU_n)$ in the case of noise-free data.
\begin{prop}\label{prop:Dirac}
Let $\PSU_{0:N}$ be a Markov chain with a transition probability defined by~\eqref{eqn:forward-transition}. Suppose that the initial state $\PSU_0$ follows a Dirac distribution at a point $\BU_0$, i.e., $\PSU_0 \sim \delta_{\BU_0}(\cdot)$, then
\[
H(\PSU_{n-1}|\PSU_n) = H(\PSU_n|\PSU_{n-1}) + H(\PSU_{n-1} | \PSU_0) - H(\PSU_n | \PSU_0) \,,\; \qquad n = 1, \ldots, N\,.
\]
\end{prop}

\begin{proof}
For $n = 1, \ldots, N$, applying the data processing inequality to the Markov chain
$\PSU_0\to \PSU_{n-1} \to \PSU_n$ we have
\begin{equation}\label{eqn:data-processing}
I(\PSU_0; \PSU_{n-1}) \geq I(\PSU_0; \PSU_{n})\,.
\end{equation}
Furthermore, from the definition of the mutual information, we have
\[
I(\PSU_0; \PSU_{n-1}) = H(\PSU_{n-1}) - H(\PSU_{n-1} | \PSU_0),
\]
\[
I(\PSU_0; \PSU_{n}) = H(\PSU_{n}) - H(\PSU_{n} | \PSU_0).
\]
Hence, we obtain in terms of conditional entropies
\[
H(\PSU_{n}) - H(\PSU_{n-1}) \leq H(\PSU_{n} | \PSU_0)  -  H(\PSU_{n-1} | \PSU_0).
\]
Owing to the fact that (from the symmetry $I(\PSU_{n-1} ; \PSU_n) = I(\PSU_n ; \PSU_{n-1})$) 
\[
H(\PSU_n) - H(\PSU_{n-1})  = H(\PSU_n|\PSU_{n-1}) - H(\PSU_{n-1}| \PSU_n),
\]
we immediately obtain the inequality
\[
H(\PSU_{n-1}|\PSU_n) \geq H(\PSU_n|\PSU_{n-1}) + H(\PSU_{n-1} | \PSU_0) - H(\PSU_n | \PSU_0)
\]
Note that the inequality~\eqref{eqn:data-processing} is attained if and only if $\PSU_0 \to \PSU_n \to \PSU_{n-1}$ also forms a Markov chain, i.e.,
\begin{equation}\label{eqn:Markov-property}
\bar{\rho}(\BU_{n-1} | \BU_n, \BU_0) = \bar{\rho}(\BU_{n-1} | \BU_n),
\end{equation}
which holds true when $\PSU_0\sim\delta_{\BU_0}(\cdot)$, where $\delta_{\BU_0}(\cdot)$ is the Dirac measure supported at $\BU_0$. 

\end{proof}

\begin{remark}
An important implication of the above proposition is the parameterization of the reverse process covariance~\eqref{eqn:cov}. 
When $\PSU_0\sim\delta_{\BU_0}(\cdot)$, the Markov property~\eqref{eqn:Markov-property} holds and hence $\bar{\rho}(\BU_{n-1} | \BU_n, \BU_0)$ is independent of $\BU_0$. 
With the parameterization~\eqref{eqn:cov} for $\Sigma_{\theta}(n, u_n, a)$ in the noise-free setting, the approximated reverse Markov chain that DDPM learned becomes exact, i.e.,
$\bar{\rho}_{\theta^*}(\BU_{n-1} | \BU_n) = \bar{\rho}(\BU_{n-1} | \BU_n)$ (see~\eqref{eqn:true-reverse-density} and~\eqref{eqn:approx-reverse-density}). Assuming $\PSV_N$ and $\PSU_N$ have the same distribution $\mathcal{N}(0, \mathbf{I})$, the two Markov chains $\PSV_{0:N}$ and $\PSU_{0:N}$ are indistinguishable and hence
their conditional entropies coincide,
\[H(\PSU_{n-1}|\PSU_n) = H_{\theta^*}(\PSV_{n-1}|\PSV_n)\,,\;\qquad n = 1, \dots, N,\]
where we have used $H_{\theta^*}$ to emphasize the fact that the conditional entropy depends on the parameter.  
\end{remark}

A similar result holds in the case when output data are perturbed by an additive
standard Gaussian noise. 
\begin{prop}\label{prop:Gaussian}
Let $\PSU_{0:N}$ be a Markov chain with transition probability defined by~\eqref{eqn:forward-transition}. Suppose that the initial state $\PSU_0$ follows a (nonzero mean) Gaussian distribution with the identity covariance, i.e., $\PSU_0 \sim \mathcal{N}(m, \mathbf{I})$, then
\[
H(\PSU_{n-1}|\PSU_n) = H(\PSU_n|\PSU_{n-1}) \,,\; \qquad n = 1, \ldots, N\,.
\]
\end{prop}
\begin{proof}
The identity follows from the well-known property that
the Gaussian distribution $\mathcal{N}(m, \mathbf{I})$ has the maximum entropy among all distributions whose density is strictly positive on 
$\R^n$ and is subject to the identity covariance constraint.
The forward transition~\eqref{eqn:forward-transition} preserves the identity covariance of the forward DDPM since 
\[
\text{Cov}(\PSU_n) = \beta_n \text{Cov}(\PSU_{n-1}) + (1 - \beta_n)\text{Cov}(\epsilon_n) = \mathbf{I}.
\]
Therefore, $H(\PSU_n) = H(\PSU_{n-1})$ for all $n = 1, \cdots, N$, which immediately implies the desired equality since
\[
H(\PSU_{n-1}) + H(\PSU_n | \PSU_{n-1})  =  H(\PSU_n) + H(\PSU_{n-1} | \PSU_n) .
\]
\end{proof}
The result suggests that the DDPM learns a transformation between $\mathcal{N}(m, \mathbf{I})$ and $\mathcal{N}(0, \mathbf{I})$, which corresponds to the optimal transport map between the two Gaussians~\cite{khrulkov2022understanding}. 
Furthermore, after integrating out $\PSU_0$ in~\eqref{eqn:true-reverse-density}, the true reverse distribution $\rho(\PSU_{n-1} | \PSU_n)$ has a covariance $\beta_n \mathbf{I}$. Hence, the parameterization~\eqref{eqn:cov} for $\Sigma_{\theta}$ in the Gaussian noise setting also leads to \[H(\PSU_{n-1}|\PSU_n) = H_{\theta^*}(\PSV_{n-1}|\PSV_n)\,,\;\qquad n = 1, \dots, N.\]

In both cases, the noise-free and Gaussian noise, the learned parametric Markov joint $\bar{\rho}_{\theta^*}$ achieves the exact likelihood under the true Markov joint $\bar{\rho}$, i.e., 
\[
\begin{split}
\mathbb{E}_{\bar{\rho}_{\theta^*}}\left[-\log \bar{\rho}_{\theta^*}(\PSV_{0:N}) \right]
&=
H_{\theta^*}(\PSV_N) + \sum_{n=1}^N H_{\theta^*}(\PSV_{n-1} | \PSV_n) \\
&=
H(\PSU_N) + \sum_{n=1}^N H(\PSU_{n-1} | \PSU_n) 
=
\mathbb{E}_{\bar{\rho}}\left[-\log \bar{\rho}(\PSU_{0:N}) \right]\,.
\end{split}
\]

%
%

\section{Gaussian random fields and Wiener processes on Hilbert spaces}\label{app:Wiener-process}
We review basic elements of probability theory in infinite-dimensional Hilbert spaces~\cite{bogachev1998gaussian, da2014stochastic}. Let $(\Omega, \mathcal{F}, \mathbb{P})$ be a probability space and $(\mathcal{U}, \langle \cdot, \cdot \rangle_{\mathcal{U}})$ be an infinite-dimensional separable Hilbert space with the inner product $\langle \cdot, \cdot \rangle_{\mathcal{U}}$. 
We call a bounded linear operator $\mathcal{C}: \mathcal{U} \to \mathcal{U}$ a covariance operator if it is self-adjoint, positive-definite and trace class. Then the spectral theorem for self-adjoint operators states that $\mathcal{C}$ admits a countable set of orthonormal eigenfunctions $\{e_j\}_{j \in \mathbb{N}}$ with corresponding positive eigenvalues $\{\lambda_j\}_{j \in \mathbb{N}}$ satisfying $\sum_j \lambda_j < \infty$. Moreover, $\{e_j\}_{j \in \mathbb{N}}$ form a complete orthonormal basis for $\mathcal{U}.$  

A random variable $U: \Omega \to \mathcal{U}$ is said to be a Gaussian random field if for every $u \in \mathcal{U}$, the $\mathbb{R}$-valued random variable $\langle U, u\rangle_{\mathcal{U}}$ is Gaussian (possibly degenerate).
We say $U$ is centered and of covariance operator $\mathcal{C}$ if 
\[
\begin{split}
&\mathbb{E}[\langle U, u \rangle_{\mathcal{U}}] = 0, \quad  u \in \mathcal{U}\,,\\
&\mathbb{E}[\langle U, u_1 \rangle_{\mathcal{U}} \langle U, u_2 \rangle_{\mathcal{U}}] = \langle \mathcal{C}u_1, u_2 \rangle_{\mathcal{U}}, \quad  u_1, u_2 \in \mathcal{U}\,.
\end{split}
\]
We denote by $\mathcal{N}(0, \mathcal{C})$ the probability measure induced by the Gaussian random field $U$ and write $U \sim \mathcal{N}(0, \mathcal{C})$.
For any Gaussian random field $U \sim \mathcal{N}(0, \mathcal{C})$, we can use the basis of eigenfunctions $\{e_j\}_{j \in \mathbb{N}}$ to write
\[
U = \sum_j \sqrt{\lambda_j} \zeta_j e_j,  
\]
where $\zeta_j$ are independent $\mathbb{R}$-valued Gaussian random variables. 
The above expansion provides the theoretical foundation for simulating a Gaussian random field. 
Moreover, if the covariance operator $\mathcal{C}$ is nondegenerate, we can define $\mathcal{C}^{\frac{1}{2}}$ by
\[
\mathcal{C}^{\frac{1}{2}} u := \sum_j \sqrt{\lambda_j} \langle u, e_j \rangle_{\mathcal{U}} e_j, \quad u \in \mathcal{U}.
\]
The operator $\mathcal{C}^{\frac{1}{2}}$ is obviously self-adjoint and of trace class (compact indeed). Moreover, it induces the Cameron-Martin space $\mathcal{M} := \mathcal{C}^{\frac{1}{2}} \mathcal{U}$ with the inner product 
\[
\langle u_1, u_2 \rangle_{\mathcal{M}} := \langle \mathcal{C}^{-\frac{1}{2}} u_1, \mathcal{C}^{-\frac{1}{2}} u_2 \rangle_{\mathcal{U}}\,, \quad u_1, u_2 \in \mathcal{M}.
\]
Since $\mathcal{C}^{-1}$ is unbounded, $\mathcal{M}$ is compactly embedded in $\mathcal{U}$. 
It should be emphasized that, when restricted to the Cameron-Martin space $(\mathcal{M}, \langle \cdot, \cdot \rangle_{\mathcal{M}})$, the Gaussian random field $U$ has an identity covariance operator $\mathcal{I}: \mathcal{M} \to \mathcal{M}$ since, for all $u_1, u_2 \in \mathcal{M}$,
\[
\mathbb{E}[\langle U, u_1 \rangle_{\mathcal{M}} \langle U, u_2 \rangle_{\mathcal{M}}]
=
\mathbb{E}[\langle U, \mathcal{C}^{-1} u_1 \rangle_{\mathcal{U}} \langle U, \mathcal{C}^{-1} u_2 \rangle_{\mathcal{U}}]
=
\langle u_1, \mathcal{C}^{-1} u_2 \rangle_{\mathcal{U}} 
=
\langle u_1, u_2 \rangle_{\mathcal{M}} .
\]

Finally, we define the Wiener process in Hilbert spaces. 
Let $\{\mathcal{F}_t\}_{t > 0}$ be a filtration.  
We call a $\mathcal{U}$-valued continuous, $\mathcal{F}_t$ adapted process $\{W^{\mathcal{U}}(t)\}_{t > 0}$ a $\mathcal{C}$-Wiener process if 
\begin{enumerate}
    \item $W^{\mathcal{U}}(0) =  0$,
    \item $W^{\mathcal{U}}(t) - W^{\mathcal{U}}(s)$ and $\mathcal{F}_s$ are independent for all $t \geq s$,
    \item $W^{\mathcal{U}}(t) - W^{\mathcal{U}}(s) \sim \mathcal{N}(0, (t-s) \mathcal{C})$ for all $t \geq s$.
\end{enumerate}
It can be shown that such a process always exists for a given Hilbert space $\mathcal{U}$. 

%
%

\section{More details on computational benchmarks}\label{sec:detailed-experiments}
We provide further details about data generation, training, and testing of PDNO for the test problems. Before we present a detailed setup for each test problem, we remind the reader of the following general setup for all test problems presented.
\begin{enumerate}[leftmargin=*]
    \item The data set is obtained by sampling inputs, i.e., the fields $a \in \mathcal{A}$ and then solving the corresponding equation $\EQOP(u;a)=0$ to obtain the output $u\in\SU$. Numerically, we only have access to their finite-dimensional counterparts $\BA$ and $\BU$.
    \item The sampling distribution $a \sim \nu_a$ generally depends on a specific model. In the tests reported here and
in \cite{FNO}, the input fields $a$ are sampled from random fields that are obtained by transformations of Gaussian fields of mean zero, with a given covariance operator described for each test separately.
    \item The training data set $\mathcal{D}=\{(\BA^{(i)},\BU^{\eta, (i)})\}_{i=1}^{M_{\text{train}}}$ is built from $M_{\text{train}}$ samples. Similarly, a test data set $\widetilde{D}$ is generated with $M_{\text{test}}$ testing pairs.
    \item Sequential sampling of the reverse Markov chain is represented by a UNet backbone whose parameters are shared across all time steps $n = 1, \ldots, N$. The sinusoidal position encoding is used to specify the diffusion time $n$.
    \item We use a multi-head self-attention block that connects the up-scaling and down-scaling blocks of the UNet. However, we found that the impact of the self-attention block is inconsequential to the final computational accuracy.
    \item Compared to unconditional DDPM, the proposed PDNO takes an additional argument $\BA$ for the noise predictor $\epsilon_{\theta}(n, \BU_n, \BA)$. Since $\BA$ and $\BU_n$ are generally of the same dimension, we simply apply channel-wise concatenation for $\BA$ and $\BU_n$ before feeding them into the UNet. However, we emphasize that the architecture can be easily modified to handle $\BA$ and $\BU_n$ that are of different spatial dimensions by projecting them into a common latent space. 
\end{enumerate}

\begin{table}[!h]
\caption{Training hyper-parameters for Test 1:
Elliptic equation in 1D}\label{tab:params-test1}
\begin{center}
\begin{tabular}{p{0.4\textwidth} p{0.5\textwidth}} 
 \hline
{\bf Architecture} & \\
 \hline
 Initial base channels   & $32$\\
 Up/down sampling multiplier    &   $(1, 2, 4, 8)$\\
 Residual blocks per resolution  &  $2$\\
 \hline
{\bf Training} & \\
  \hline
  Timesteps $N$    &   $100$\\
  Epochs    &   $1000$  \\
  Batch size &  $50$    \\
  Training data size $M_{\text{train}}$  & $10000$\\
  Noise schedule   &   cosine\\
  Optimizer   &    Adam \\
  Learning rate  &  $1$e$-4$ initially; halved by half per $100$ epochs\\
    Data normalized  &    No\\
  \hline
{\bf Testing} &  \\
\hline
  Testing data size $M_{\text{test}}$  &  $10000$\\
  \# of inference samples $M_{\text{infer}}$   &  $1$ if noise-free; $500$ if noisy\\
\end{tabular}
\end{center}
\end{table}

\begin{table}[!h]
\caption{Training hyper-parameters for Test 2:
Elliptic equation in 2D}\label{tab:params-test2}
\begin{center}
\begin{tabular}{p{0.4\textwidth} p{0.5\textwidth}} 
 \hline
{\bf Architecture} & \\
 \hline
 Initial base channels   & $64$\\
 Up/down sampling multiplier    &   $(1, 2, 4, 8)$\\
 Residual blocks per resolution  &  $2$\\
 \hline
{\bf Training} & \\
  \hline
  Timesteps $N$    &   $100$\\
  Epochs    &   $1000$  \\
  Batch size &  $20$    \\
  Training data size $M_{\text{train}}$  & $1000$\\
  Noise schedule   &   cosine\\
  Optimizer   &    Adam \\
  Learning rate  &  $1$e$-4$ initially; halved by half per $100$ epochs\\
    Data normalized  &    Yes\\
  \hline
{\bf Testing} &  \\
\hline
    Testing data size $M_{\text{test}}$  &  $100$\\
  \# of inference samples $M_{\text{infer}}$   &  $1$ if noise-free; $500$ if noisy\\
 \hline
\end{tabular}
\end{center}
\end{table}



\begin{table}[!h]
\caption{Training hyper-parameters for Test 3:
Burgers' equation}\label{tab:params-test3}
\begin{center}
\begin{tabular}{p{0.4\textwidth} p{0.5\textwidth}} 
 \hline
{\bf Architecture} & \\
 \hline
 Initial base channels   & $32$\\
 Up/down sampling multiplier    &   $(1, 2, 4, 8)$\\
 Residual blocks per resolution  &  $2$\\
 \hline
{\bf Training} & \\
  \hline
  Timesteps $N$    &   $100$\\
  Epochs    &   $1000$  \\
  Batch size &  $32$    \\
  Training data size $M_{\text{train}}$  & $1024$\\
  Noise schedule   &   linear\\
  Optimizer   &    Adam \\
  Learning rate  &  $1$e$-4$ initially; halved by half per $100$ epochs\\
    Data normalized  &    No\\
  \hline
{\bf Testing} &  \\
\hline
Testing data size $M_{\text{test}}$  &  $1024$\\
  \# of inference samples $M_{\text{infer}}$   &  $1$ if noise-free\\
 \hline
\end{tabular}
\end{center}
\end{table}

\begin{table}[!h]
\caption{Training hyper-parameters for Test 4:
Advection equation}\label{tab:params-test4}
\begin{center}
\begin{tabular}{p{0.4\textwidth} p{0.5\textwidth}} 
 \hline
{\bf Architecture} & \\
 \hline
 Initial base channels   & $32$\\
 Up/down sampling multiplier    &   $(1, 2, 4, 8)$\\
 Residual blocks per resolution  &  $2$\\
 \hline
{\bf Training} & \\
  \hline
  Timesteps $N$    &   $100$\\
  Epochs    &   $1000$  \\
  Batch size &  $20$    \\
  \# of training points $M_{\text{train}}$ & $1500$ for noise free case; $10^4$ for noisy case\\
  Noise schedule   &   cosine\\
  Optimizer   &    Adam \\
  Learning rate  &  $8$e$-5$ fixed\\
  Data normalized  &    No for noise free case; Yes for noisy case \\
  \hline
{\bf Testing} &  \\
\hline
  Testing data size $M_{\text{test}}$  &  $300$\\
  \# of inference samples $M_{\text{infer}}$   &  $1$ for noise free case; $10^4$ for noisy case\\
 \hline
\end{tabular}
\end{center}
\end{table}

%
%
\section{Table of notations}
For quick reference, we summarize important notations in Table~\ref{tab:notations}.
\begin{table}[!h]
\caption{Table of notations}\label{tab:notations}
\begin{center}
\begin{tabular}{p{0.2\textwidth} p{0.8\textwidth}} 
 \hline
{\bf Notation} & {\bf Definition} \\
 \hline
$N$  & Number of steps for the discrete time Markov chains\\
$T$  & Terminal time for the continuous time Markov processes\\
 $\mathcal{G}$   &  The mapping that defines the abstract PDE in~\eqref{eqn:random-PDE}\\
  $\SA$   &    Function spaces of input functions $a$ \\
  $\SU$   &    Function spaces of output functions $u$ \\ 
  $\mathbb{R}^{d_{\mathcal{A}}}$   &    Space for the discretization mesh for $a$ \\
  $\mathbb{R}^{d_{\mathcal{U}}}$   &    Space for the discretization mesh for $u$ \\ 
  $D$    & The physical space of the abstract PDE in~\eqref{eqn:random-PDE}\\
 $a$     & Input function in $\SA$\\
 $u$     & Output function in $\SU$\\
  $\BA$     & Finite-dimensional projection of $a$\\
 $\BU$     & Finite-dimensional projection of $u$\\
   $\eta$ & Additive noise that pollutes the output $\BU$\\ 
  $\BU^{\eta}$     & Noisy observation of $\BU$\\
  $\nu_a$  & The probability measure of the input function $a \in \mathcal{A}$\\
  $\nu_{\BA}$  & The pushforward measure of the projected input vector $\BA \in \mathbb{R}^{d_{\mathcal{A}}}$\\
  $\SOP$   & The true infinite-dimensional operator\\
  $\SOPA$  & The approximated infinite-dimensional operator \\
  $s_{\theta}$ & The finite-dimensional approximation to $\SOP$ \\
  $q_{\theta}(\cdot| \BA)$   &  The parametric conditional model for operator learning\\
  $p(\cdot | \BA)$            & The true conditional model for operator learning\\  
  $U(t; a)$    & The continuous time forward time Markov process\\
  $V(t; a)$    & The parameterized continuous time reverse time Markov process\\
  $U_n(\BA)$    & The discrete time forward time Markov chain\\
  $V_n(\BA)$    & The parameterized discrete time reverse time Markov chain\\
  $\rho_t(\cdot | \BA)$     & The density of $U(t; \BA)$\\
  $\bar{\rho}(\cdot| \BA)$   &  The path-space joint density of the Markov chain $\PSU_{0:N}(\BA)$ \\
  $\bar{\rho}_{\theta}(\cdot | \BA)$   &  The path-space joint density of the parameterized Markov chain $\PSV_{0:N}(\BA)$ \\
  $\delta_{\BU}(\cdot|\BA)$   & The conditional Dirac measure centered at $\BU$\\
  $\epsilon_n$     & Standard Gaussian noise\\
  $\epsilon_{\theta}$     & Noise approximate\\
  $\mathcal{D}$   &  Noise free training data set of input-output pairs\\
  $\mathcal{D}_{\eta}$   & Noisy training data set whose outputs are corrupted by additive noise\\
  $\widetilde{\mathcal{D}}$   & Noise free testing data set\\
  $M_{\text{train}}$   &   Size of training data, i.e., $|\mathcal{D}|$ or $|\mathcal{D}_{\eta}$|\\
  $M_{\text{test}}$   &   Size of testing data, i.e., $|\widetilde{\mathcal{D}}|$\\
  $M_{\text{infer}}$   &   Number of samples from inference\\
  \hline
\end{tabular}
\end{center}
\end{table}

\end{document}